\documentclass[11pt,reqno]{amsart}

\usepackage[T1]{fontenc}
\usepackage[utf8]{inputenc}
\usepackage{lmodern}
\usepackage{microtype}
\usepackage{amsmath,amssymb,amsthm,mathtools}
\usepackage[hidelinks]{hyperref}
\usepackage[nameinlink,noabbrev]{cleveref}
\usepackage{newtxtext,newtxmath}
\usepackage{tikz}
\usetikzlibrary{calc}

\hypersetup{
  pdftitle={Boundary Algorithms for Thompson Groups},
  pdfauthor={Ayberk Zeytin},
  pdfsubject={Farey boundary flips, effective decompositions, and Thompson groups},
  pdfkeywords={modular group, Farey tree, Thompson group, continued fraction, flip, shuffle, algorithm}
}

\tikzset{
  farey edge/.style={
    draw=black,
    line width=0.9pt,
    line cap=round,
    line join=round
  },
  open vertex/.style={
    circle,
    draw=black,
    fill=white,
    line width=0.9pt,
    minimum size=6.2pt,
    inner sep=0pt
  },
  filled vertex/.style={
    circle,
    draw=black,
    fill=black!72,
    line width=0.8pt,
    minimum size=6.2pt,
    inner sep=0pt
  },
  edge label/.style={
    fill=white,
    inner sep=1.5pt,
    font=\small
  }
}

\numberwithin{equation}{section}

\newtheorem{theorem}{Theorem}[section]
\newtheorem{proposition}[theorem]{Proposition}
\newtheorem{lemma}[theorem]{Lemma}
\newtheorem{corollary}[theorem]{Corollary}
\theoremstyle{definition}
\newtheorem{definition}[theorem]{Definition}
\newtheorem{example}[theorem]{Example}
\theoremstyle{remark}
\newtheorem{remark}[theorem]{Remark}

\DeclareMathOperator{\psl}{\mathrm{PSL}_2(\Z)}
\DeclareMathOperator{\pgl}{\mathrm{PGL}_2(\Z)}
\DeclareMathOperator{\gl}{\mathrm{GL}_2(\Z)}
\DeclareMathOperator{\sll}{\mathrm{SL}_2(\Z)}
\newcommand{\Z}{\mathbf{Z}}
\newcommand{\Q}{\mathbf{Q}}
\newcommand{\R}{\mathbf{R}}

\newcommand{\Hh}{\mathbb{H}}
\newcommand{\PP}{\mathbb{P}}
\newcommand{\F}{\mathcal{F}}
\newcommand{\T}{\mathsf{T}}

\newcommand{\cyl}[1]{\mathcal{C}(#1)}
\newcommand{\ev}{\operatorname{ev}}
\newcommand{\id}{\operatorname{id}}

\title[Boundary Algorithms for Thompson Groups]
{Boundary Algorithms for Thompson Groups}

\author{Ayberk Zeytin}
\address{Department of Mathematics, Galatasaray University, Istanbul, T\"urkiye}


\begin{document}

\begin{abstract}
	We construct a deterministic Farey-boundary algorithm converting a cyclic prefix table for an element of $\mathsf T$ into a word in addressed boundary flips. The reduction uses rational boundary pairs and cyclic order, with a marked edge determining the residual modular correction. Its unmarked stage uses at most twice the minimum number of diagonal exchanges. For input size $s$, the complete output has $O(s)$ factors, encoded length $O(s^2)$, and running time $O(s^3)$. Adjoining shuffles extends the construction to $\mathsf V$ with polynomial bounds. We also give explicit boundary generating sets and their local flip--shuffle relations. Changes of generators in the Lochak--Schneps and Bleak--Quick presentations yield a two-flip presentation of $\mathsf T$ and a presentation of $\mathsf V$ with two flips, one shuffle, and nine relations.
\end{abstract}

\maketitle

\section{Introduction}
\label{sec:intro}

Thompson's groups $\mathsf F$, $\mathsf T$, and $\mathsf V$ are central examples in combinatorial group theory, dynamics, and low-dimensional topology. In their prefix-code description, elements act on the binary Cantor space by finite prefix replacements. The correspondence between prefixes preserves linear order for $\mathsf F$ and cyclic order for $\mathsf T$, while no order-preservation condition is imposed for $\mathsf V$.  Equivalently, these groups may be represented by tree-pair diagrams; see \cite{CannonFloydParry1996,CannonFloyd2011}.  We use prefix-code tables as the finite input model throughout this paper.

A second classical description of $\mathsf T$ arises from the Farey tessellation, on which the modular group $\psl$ acts. Replacing one diagonal of an ideal quadrilateral by the other gives an elementary flip of the Ptolemy groupoid. Penner developed the corresponding constructions in \cite{Penner1993,Penner1997}; the resulting Ptolemy group was identified with Thompson's group $\mathsf T$ in \cite{Imbert1997}, and its groupoid form and finite presentation were further developed in \cite{LochakSchneps1997}. Dehornoy obtained presentations of $\mathsf F$ and $\mathsf V$ using local transformations of parenthesized expressions: changes of parentheses for $\mathsf F$, and changes of parentheses together with exchanges of subexpressions for $\mathsf V$ \cite{Dehornoy2005}. 

We present a deterministic algorithm that converts a cyclic prefix-code table for an element of $\mathsf T$ into a word in addressed Farey-boundary flips. The reduction uses rational boundary pairs and cyclic order, while a transported oriented edge determines the final modular correction. We bound both the fully encoded output length and the running time. Adjoining shuffles extends this procedure to $\mathsf V$. The underlying flip connectivity and polygonal fan reduction are classical \cite{SleatorTarjanThurston1988}. Related constructions include the Mosher-type normal forms and combing for $\mathsf T$ developed by Funar and Kapoudjian \cite{FunarKapoudjian2011}, following Mosher's work on mapping class groups \cite{Mosher1995}, together with tree-pair and prefix-table normal forms and metric algorithms for Thompson groups \cite{CannonFloydParry1996,Fordham2003, BurilloClearySteinTaback2009,Birget2004}.

The Farey boundary connects the finite prefix description with the geometric reduction. Binary coding identifies the ordinary boundary with the binary Cantor space, where flips and shuffles act by explicit prefix replacements. Identifying the two ends representing each rational point gives the smashed boundary, which continued-fraction evaluation identifies with $\PP^1(\R)$. On this quotient, rational boundary pairs encode the configurations used by the flip algorithm, and cyclic order determines which pairs are linked. The distinguished oriented edge retains the information needed to recover the group element from the reduced configuration. These coordinates identify the flip group with $\mathsf T$, the flip--shuffle group with $\mathsf V$, and the group generated by flips supported in a fixed proper cylinder with $\mathsf F$. 

Our main theorem, \cref{thm:smashed-boundary-algorithm},
gives quantitative bounds for this construction. For a prefix table of size $s$, measured by its number of entries and total prefix length, the algorithm returns $O(s)$ factors in the full address-indexed flip family. Writing all supporting addresses gives encoded length $O(s^2)$, and the running time is $O(s^3)$. With the ordering conventions fixed, the output depends only on the represented group element. The unmarked reduction uses at most twice the minimum number of diagonal exchanges; this comparison excludes the final modular correction.

For $\mathsf V$, we combine the $\mathsf T$ algorithm with a decomposition of the finite cylinder permutation into conjugates of a basic shuffle. This yields polynomial bounds for the output length and running time. The cylinder-permutation factorization is standard \cite{CannonFloydParry1996,BleakQuick2017}; the construction supplies explicit flip--shuffle words by applying the boundary reduction to its cyclic factors.

The boundary realization gives the generating sets
\[
\mathsf T=\langle\phi_I,\phi_{SL^2}\rangle,
\qquad
\mathsf V=\langle\phi_I,\phi_{SL^2},\sigma_S\rangle.
\]
We describe the local relations among addressed flips and shuffles and derive finite presentations in these boundary generators. For $\mathsf V$, the resulting presentation has nine defining relations. Completeness follows by explicit changes of generators from the presentations of Lochak--Schneps \cite{LochakSchneps1997} and Bleak--Quick \cite{BleakQuick2017}. The contribution is their Farey-boundary formulation, compatible with the word-extraction algorithms developed here.

The paper is organized as follows.
Section~\ref{sec:thompson-prefix} recalls the simultaneous prefix-code
definitions of $\mathsf F$, $\mathsf T$, and $\mathsf V$.
Section~\ref{sec:modular-case} develops the Farey tree, its smashed
boundary, and the continued-fraction coordinates.
Section~\ref{sec:boundary-operators} defines boundary flips and shuffles,
identifies the groups they generate, and records their local relations.
Section~\ref{sec:effective-decompositions} gives the effective
decomposition algorithms and their complexity bounds for
$\mathsf T$ and $\mathsf V$.  Finally,
Section~\ref{sec:boundary-presentations} derives the boundary presentation
theorems as structural applications of the preceding dictionary.

\section{Thompson groups as prefix replacements}
\label{sec:thompson-prefix}

Let $\mathcal B=\{0,1\}^{\mathbb N}$ be the binary Cantor space, and let $<$ denote the
lexicographic order on finite binary words. For
a finite binary word $u$, write $u\mathcal{B}:=\{u\xi:\xi\in\mathcal{B}\}$. A finite set $P$ of binary words is a \emph{finite maximal prefix code} (also called a \emph{complete prefix code}) if the cylinders $u\mathcal{B}$, $u\in P$, are pairwise disjoint and cover
$\mathcal{B}$. 
If
\[
 P=\{p_1<\cdots<p_n\},
 \qquad
 Q=\{q_1<\cdots<q_n\}
\]
are finite maximal prefix codes and $\tau\in\operatorname{Sym}(n)$,
define the prefix-replacement homeomorphism
\[
 h_{P,Q,\tau}(p_i\xi)=q_{\tau(i)}\xi,
 \qquad i=1,\ldots,n,\quad \xi\in\mathcal B.
\]
We call the triple \((P,Q,\tau)\), or equivalently the correspondence \(p_i\mapsto q_{\tau(i)}\), a \emph{prefix-code table} for \(h_{P,Q,\tau}\).
Let $\mathscr C_n$ denote the set of finite maximal prefix codes with $n$
elements.  Allowing $n$ and the two codes to vary, Thompson's groups may be
defined simultaneously by
\[
 \mathsf V
 =
 \left\{
 h_{P,Q,\tau}:n\geq1,\ P,Q\in\mathscr C_n,\
 \tau\in\operatorname{Sym}(n)
 \right\},
\]
\[
 \mathsf T
 =
 \left\{
 h_{P,Q,\tau}:n\geq1,\ P,Q\in\mathscr C_n,\
 \tau(i)\equiv i+k\pmod n
 \text{ for some }k\in\mathbb Z/n\mathbb Z
 \right\},
\]
and
\[
 \mathsf F
 =
 \left\{
 h_{P,Q,\operatorname{id}}:n\geq1,\ P,Q\in\mathscr C_n
 \right\}.
\]
For $\mathsf T$, the permutation of the ordered prefix lists must be a
cyclic shift, not an arbitrary permutation consisting of a single cycle.
Indices are read modulo $n$, with values in $\{1,\ldots,n\}$.

An element generally has more than one prefix-code table.  A simultaneous
expansion replaces one row $p_i\mapsto q_{\tau(i)}$ by the two rows
\[
 p_i0\mapsto q_{\tau(i)}0,
 \qquad
 p_i1\mapsto q_{\tau(i)}1.
\]
It does not change the resulting homeomorphism, and any two tables for the
same element have a common expansion. Also, simultaneous expansions preserve
linear and cyclic order. Conversely, sibling rows \(p_i0\mapsto q_{\tau(i)}0\) and \(p_{i}1\mapsto q_{\tau(i)}1\) may be merged to \(p_{i}\mapsto q_{\tau(i)}\); this is called a reduction. Every prefix-code table admits a unique reduced table, equivalently a unique table with the minimum number of rows, and every other table is obtained from it by expansion.   
Composition is obtained after passing to a common refinement;
thus the three sets are groups and the inclusions
$\mathsf F<\mathsf T<\mathsf V$ are immediate.  This is the standard
prefix-code, or tree-pair, description of the Thompson groups
\cite[\S\S4--6]{CannonFloydParry1996}; see also
\cite[\S2.1]{BleakQuick2017} for the prefix-exchange model of $\mathsf V$.

\begin{remark}
	Identifying the two binary expansions of each dyadic point, and also
	the endpoints in the circle case, recovers the standard piecewise-linear
	realizations of $\mathsf F$ on $[0,1]$ and $\mathsf T$ on $\R/\Z$;
	see \cite{CannonFloydParry1996,CannonFloyd2011}.
\end{remark}

\section{The Farey tree and its boundary coordinates}
\label{sec:modular-case}

We construct the Farey tree from the standard action of $\psl$ on the upper
half-plane $\Hh$ and identify its smashed boundary with $\PP^1(\R)$.
Standard continued fractions then provide the word coordinates used
throughout the paper.

\subsection{The Farey tree and its Cantor boundary}
\label{sec:farey-tree}
%

We use the projective action
\[
g\cdot z=\frac{az+b}{cz+d},
\qquad
g=\begin{pmatrix}a&b\\c&d\end{pmatrix}\in\gl,
\]
with the usual conventions at $\infty$, and the matrices
\begin{equation}\label{eq:ULS}
	U=\begin{pmatrix}0&1\\1&0\end{pmatrix},\qquad
	L=\begin{pmatrix}1&-1\\1&0\end{pmatrix},\qquad
	S=\begin{pmatrix}0&-1\\1&0\end{pmatrix}.
\end{equation}
Put $T=\left(\begin{smallmatrix}1&1\\0&1\end{smallmatrix}\right)
=-LS$.
The modular group $\psl=\sll/\{\pm I\}$ preserves
$\Hh$ and has presentation
\[
\psl=\langle S,L\mid S^2=L^3=1\rangle
\cong(\Z/2\Z)*(\Z/3\Z).
\]

The element $S$ fixes $i$, labelled by $\circ$, while $L$ fixes
$\zeta=e^{\pi i/3}$, labelled by $\bullet$.  Label the geodesic segment joining
$i$ to $\zeta$ by $e_I$.  The Farey tree $\F$ is the union of the
$\psl$-translates of $\circ$, $\bullet$ and $e_I$: its white and black vertices are the respective
orbits of $\circ$ and $\bullet$, and the edge $W\cdot e_I$ is denoted by
$e_W$.  By construction, $\F$ is a bipartite tree.  The orientation of
$\Hh$ equips $\F$ with a ribbon structure, hence with a cyclic order at every
black vertex.

The \emph{Cantor boundary} of $\F$ based at $I$, denoted
$\partial_I\F$, consists of infinite non-backtracking paths
starting at the distinguished edge.  Every end has a unique
turn-word representation
\[
\varepsilon=S^\delta C_1C_2\cdots,
\qquad
\delta\in\{0,1\},\qquad C_j\in\{LS,L^2S\}.
\]
Here $\delta$ specifies the initial side of the distinguished
edge, and each $C_j$ specifies a subsequent forward turn.

\subsection{Continued fractions and the smashed boundary}

The tree $\F$ has two ends corresponding to each rational number: one eventually turns only to the left and the other eventually turns only to the right.  Indeed, finding the standard continued fraction expansion of a real number $x$ can be encoded in terms of the group $\pgl $ :
\begin{align*}
	x  & = a_{0} + \cfrac{1}{a_{1}+\cfrac{1}{a_{2} + \cfrac{1}{a_{3} + \cfrac{1}{\ddots}}}} \\
						& = T^{a_{0}} U \left({a_{1}+\cfrac{1}{a_{2} + \cfrac{1}{a_{3} + \cfrac{1}{\ddots}}}}\right)  \\
						& = T^{a_{0}} U T^{a_{1}} U \cdots
\end{align*}
The real number $x \in \Q$ if and only if its standard continued fraction expansion is finite. In that case,
\[ 
	x = \frac{p}{q} = T^{a_{0}} U T^{a_{1}} U \ldots T^{a_{k}} (0) = T^{a_{0}} U T^{a_{1}} U \ldots T^{a_{k}} U( \infty) 
\]
The identities $T=LS$ and $UTU=L^2S$ in $\pgl$, which give
$UT^aU=(L^2S)^a$, allow us to write any $x\in\Q$ as
\begin{align*}
	x = 
	\begin{cases}
		\left((LS)^{a_{0}}(L^2S)^{a_{1}}\right) \cdots \left((LS)^{a_{k-1}}(L^{2}S)^{a_{k}}\right)  (\infty )\\
		\left((LS)^{a_{0}}(L^2S)^{a_{1}}\right) \cdots \left((LS)^{a_{k-1}}(L^{2}S)^{a_{k} - 1}\right) (LS) (0)
	\end{cases}
\end{align*}
when the last index $k$ of the continued fraction is odd. When $k$ is even:
\begin{align*}
	x = 
	\begin{cases}
		\left((LS)^{a_{0}}(L^2S)^{a_{1}}\right) \cdots \left((LS)^{a_{k}}\right) (0) \\
		\left((LS)^{a_{0}}(L^2S)^{a_{1}}\right) \cdots \left((LS)^{a_{k} - 1}(L^{2}S)\right) (\infty)
	\end{cases}
\end{align*}
Replacing $0$ by $(L^{2}S)^{\infty}$ and $\infty$ by $(LS)^{\infty}$, we see that the planar tree has two ends corresponding to each element of $\Q \cup \{\infty\}$: one eventually turns only to the left and the other eventually turns only to the right. We therefore identify these two ends and obtain the \emph{smashed boundary} of the Farey tree.
	
Every end determines a nested sequence of closed Farey intervals.
Their intersection consists of a single point, defining the
continuous evaluation map
\begin{equation}\label{eq:ev-homeo}
	\ev:\partial_I\F\longrightarrow\PP^1(\R).
\end{equation}
The nested Farey subdivisions show that $\ev$ is surjective,
one-to-one over the irrationals, and two-to-one over
$\PP^1(\Q)$.  Each rational fibre consists of two eventually
constant ends approaching the same point from opposite sides.

\begin{definition}\label{def:smashed}
	The \emph{smashed boundary} is
	\[
	\partial_I^\#\F:=\partial_I\F/\!\sim,
	\qquad
	\varepsilon\sim\eta
	\ \Longleftrightarrow\
	\ev(\varepsilon)=\ev(\eta).
	\]
\end{definition}

Evaluation therefore induces a continuous bijection
\begin{equation}\label{eq:ev-bar-homeo}
	\overline{\ev}:\partial_I^\#\F
	\xrightarrow{\ \cong\ }\PP^1(\R).
\end{equation}
It is a homeomorphism because its domain is compact and its
codomain is Hausdorff.

\section{Boundary flips, shuffles, and Thompson groups}
\label{sec:boundary-operators}

We now define the two elementary boundary operations, describe them in
projective and binary coordinates, and identify the groups they generate.

\subsection{Boundary flips and shuffles}
\label{sec:flips-shuffles}

Put
\[
 A:=LS,\qquad B:=L^2S.
\]
Thus $A$ and $B$ are the left and right forward turns, respectively, at a
black vertex of $\F$.  
An \emph{admissible prefix} is a finite initial word of the form
\[
W=S^\delta C_1\cdots C_k,
\qquad
\delta\in\{0,1\},\quad k\geq0,\quad C_i\in\{A,B\}.
\]
Its \emph{cylinder} is
\[
\cyl{W}
:=
\left\{
S^\delta C_1\cdots C_kD_1D_2\cdots:
D_i\in\{A,B\}
\right\}.
\]
Thus $\cyl{I}$ consists of ends whose initial side is represented
by $\delta=0$, while $\cyl{S}$ consists of those with $\delta=1$,
and
\[
\partial_I\F=\cyl{I}\sqcup\cyl{S}.
\]
In particular, $\cyl{I}$ is a proper cylinder.
In the formulas below, $Y$ denotes an infinite suffix compatible
with the displayed prefix, and no cancellation is allowed.

\begin{definition}\label{def:left-flip}
On the set of ends of the Farey tree, the \emph{flip at $WL$} is defined as the map
\[
 \phi_{WL}(X)=
 \begin{cases}
  WB^2Y,&X=WBY,\\
  WBAY,&X=WABY,\\
  WAY,&X=WA^2Y,\\
  X,&X\notin\cyl{W}.
 \end{cases}
\]
Its inverse is the \emph{flip at $WL^2$}:
\[
 \phi_{WL^2}(X)=
 \begin{cases}
  WBY,&X=WBBY,\\
  WABY,&X=WBAY,\\
  WAAY,&X=WAY,\\
  X,&X\notin\cyl{W}.
 \end{cases}
\]
\end{definition}
The local picture is shown in \cref{fig:flip}.

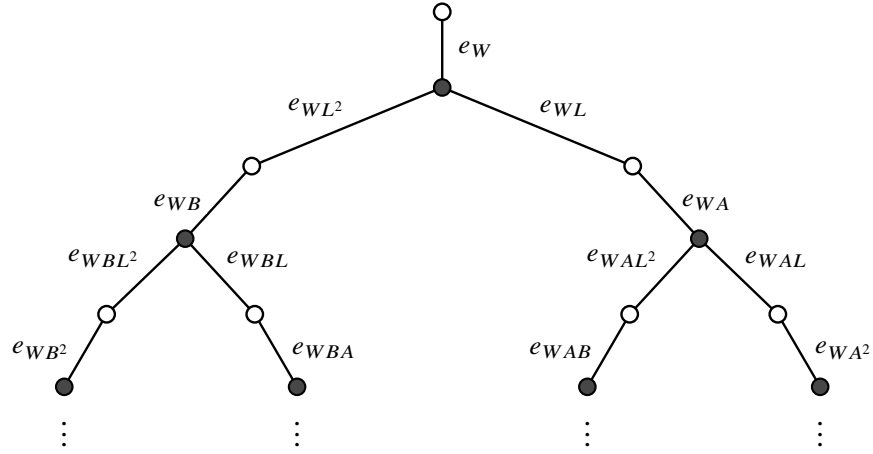
\begin{figure}[ht]
  \centering
  \begin{tikzpicture}[x=0.8cm,y=0.8cm]
  \coordinate (top) at (0,5.8);
  \coordinate (root) at (0,4.55);

  \coordinate (left-open)  at (-3.15,3.25);
  \coordinate (right-open) at ( 3.15,3.25);
  \coordinate (left-filled)  at (-4.25,2.05);
  \coordinate (right-filled) at ( 4.25,2.05);

  \coordinate (ll-open) at (-5.55,0.80);
  \coordinate (lr-open) at (-3.10,0.80);
  \coordinate (rl-open) at ( 3.10,0.80);
  \coordinate (rr-open) at ( 5.55,0.80);

  \coordinate (ll-end) at (-6.25,-0.40);
  \coordinate (lr-end) at (-2.40,-0.40);
  \coordinate (rl-end) at ( 2.40,-0.40);
  \coordinate (rr-end) at ( 6.25,-0.40);

  \draw[farey edge] (top) -- (root)
    node[edge label,midway,right=4pt] {$e_W$};

  \draw[farey edge] (root) -- (left-open)
    node[edge label,pos=0.47,above left=1pt] {$e_{WL^{2}}$};
  \draw[farey edge] (root) -- (right-open)
    node[edge label,pos=0.47,above right=1pt] {$e_{WL}$};

  \draw[farey edge] (left-open) -- (left-filled)
    node[edge label,midway,left=4pt] {$e_{WB}$};
  \draw[farey edge] (right-open) -- (right-filled)
    node[edge label,midway,right=4pt] {$e_{WA}$};

  \draw[farey edge] (left-filled) -- (ll-open)
    node[edge label,pos=0.48,above left=1pt] {$e_{WBL^{2}}$};
  \draw[farey edge] (left-filled) -- (lr-open)
    node[edge label,pos=0.48,above right=1pt] {$e_{WBL}$};
  \draw[farey edge] (right-filled) -- (rl-open)
    node[edge label,pos=0.48,above left=1pt] {$e_{WAL^{2}}$};
  \draw[farey edge] (right-filled) -- (rr-open)
    node[edge label,pos=0.48,above right=1pt] {$e_{WAL}$};

  \draw[farey edge] (ll-open) -- (ll-end)
    node[edge label,midway,left=4pt] {$e_{WB^{2}}$};
  \draw[farey edge] (lr-open) -- (lr-end)
    node[edge label,midway,right=4pt] {$e_{WBA}$};
  \draw[farey edge] (rl-open) -- (rl-end)
    node[edge label,midway,left=4pt] {$e_{WAB}$};
  \draw[farey edge] (rr-open) -- (rr-end)
    node[edge label,midway,right=4pt] {$e_{WA^{2}}$};

  \node[open vertex]   at (top) {};
  \node[filled vertex] at (root) {};

  \node[open vertex] at (left-open) {};
  \node[open vertex] at (right-open) {};
  \node[filled vertex] at (left-filled) {};
  \node[filled vertex] at (right-filled) {};

  \node[open vertex] at (ll-open) {};
  \node[open vertex] at (lr-open) {};
  \node[open vertex] at (rl-open) {};
  \node[open vertex] at (rr-open) {};

  \node[filled vertex] at (ll-end) {};
  \node[filled vertex] at (lr-end) {};
  \node[filled vertex] at (rl-end) {};
  \node[filled vertex] at (rr-end) {};

  \node[font=\large] at ($(ll-end)+(0,-0.65)$) {$\vdots$};
  \node[font=\large] at ($(lr-end)+(0,-0.65)$) {$\vdots$};
  \node[font=\large] at ($(rl-end)+(0,-0.65)$) {$\vdots$};
  \node[font=\large] at ($(rr-end)+(0,-0.65)$) {$\vdots$};
\end{tikzpicture}
  \caption{The local picture of $\F$ near a flipped edge $e_W$ with $W \neq I$.}
  \label{fig:flip}  
\end{figure}

At the base edge the two sides of $e_I$ must also be permuted. Therefore, the \emph{flip at $I$} needs to be defined separately, as follows:
\begin{equation}\label{eq:base-flip}
 \phi_I(X)=
 \begin{cases}
  SL^2Y,&X=LY,\\
  LY,&X=L^2Y,\\
  L^2Y,&X=SLY,\\
  SLY,&X=SL^2Y.
 \end{cases}
\end{equation}

These boundary flips should not be confused with involutive diagonal
exchanges in an ideal quadrilateral.  If $W$ is admissible, then
\[
 \phi_{WL}^{\,n}\bigl(WBA^\infty\bigr)=WB^{n+1}A^\infty,
 \qquad n\geq0.
\]
The ends on the right are pairwise distinct, so $\phi_{WL}$ and its inverse
$\phi_{WL^2}$ have infinite order.  Only the distinguished flip is
exceptional: $\phi_I^4=\id$.

\begin{proposition}\label{prop:flip-descends}
Every flip descends to an orientation-preserving homeomorphism of $\partial_I^{\#}\F$.  
\end{proposition}

\begin{proof}
	The nontrivial identifications are
	\[
	PAB^\infty\sim PBA^\infty
	\quad(P\text{ admissible}),
	\qquad
	B^\infty\sim SA^\infty,
	\qquad
	A^\infty\sim SB^\infty.
	\]
	The last two pairs represent $0$ and $\infty$, respectively.
	The defining substitutions of each flip and its inverse
	preserve these identifications, so they induce mutually inverse
	homeomorphisms of the quotient.
	
	Within its supporting cylinder, a left flip preserves
	the order of the prefix lists:
	\[
	(WB,WAB,WAA)\longmapsto(WBB,WBA,WA).
	\]
	The base flip cyclically permutes its four initial
	cylinders. Thus every flip preserves cyclic order
	on the quotient.
\end{proof}

\begin{definition}\label{def:shuffle}
On the set of ends of the Farey tree, the \emph{shuffle at $W$} is defined as the local transposition $\sigma_W$ given by
\[
 \sigma_W(WAY)=WBY,\qquad
 \sigma_W(WBY)=WAY,
\]
and $\sigma_W(X)=X$ for $X\notin\cyl{W}$.
\end{definition}

Every shuffle is an involutive homeomorphism of the ordinary Cantor boundary $\partial_I\F$.
It exchanges the left and right forward subtrees at $W$ and reverses their
local cyclic order.  In contrast to a flip, it need not preserve the paired rational ends.

\paragraph{Projective description.}
For an admissible prefix $P$, put
\[
I(P):=P((0,\infty))\subset\PP^1(\R).
\]
This is the open Farey arc corresponding to $\cyl{P}$,
with endpoints $P(0)$ and $P(\infty)$. Under the evaluation
map, a prefix replacement \(PY\mapsto QY\) is represented on \(I(P)\) by
\begin{equation}\label{eq:prefix-projective}
	x\longmapsto QP^{-1}\cdot x.
\end{equation}
Consequently, every boundary flip is piecewise given by elements of
\(\psl\), with finitely many rational breakpoints, and descends to a
homeomorphism of \(\PP^1(\R)\).  The same formula describes a shuffle on
the interiors of its two supporting intervals, but its one-sided values at
their rational endpoints may differ; hence a shuffle need not descend to
the smashed boundary.

\subsection{Binary coding and the question-mark conjugacy}
\label{sec:pl-circle}

We use the binary convention
\begin{equation}\label{eq:binary-turns}
	B=L^2S\longleftrightarrow0,\qquad
	A=LS\longleftrightarrow1,
\end{equation}
and encode the two initial sides of the distinguished edge $e_I$ by $0$ and $1$.
For an end $\varepsilon=S^\delta C_1C_2\cdots$, define
\begin{equation}\label{eq:farey-binary-code}
	\kappa(\varepsilon)=\delta c_1c_2\cdots,
	\qquad
	c_j=
	\begin{cases}
		0,&C_j=B,\\
		1,&C_j=A.
	\end{cases}
\end{equation}
The \emph{binary code} of an admissible prefix $W=S^\delta C_1\cdots C_k$ is the word $w=\delta c_1\cdots c_k$, so that $\kappa(\cyl{W})=w\mathcal B$. The identification induced by $\kappa$ gives
\begin{equation} 
	L^2\xi\longleftrightarrow00\xi,\qquad
	L\xi\longleftrightarrow01\xi,\qquad
	SL^2\xi\longleftrightarrow10\xi,\qquad
	SL\xi\longleftrightarrow11\xi,
\end{equation}
as displayed in \cref{fig:kappa-coding}.  
\begin{figure}[ht]
 \centering
 \begin{tikzpicture}[
   x=2.35cm,
   y=1.05cm,
   code label/.style={
     fill=white,
     inner sep=1.2pt,
     font=\scriptsize\bfseries
   },
   word label/.style={
     fill=white,
     inner sep=1.2pt,
     font=\scriptsize
   },
   farey geodesic/.style={
     draw=black!42,
     densely dashed,
     line width=0.45pt
   }
 ]
 \draw[black!55,line width=0.45pt] (-2.18,0)--(2.18,0);
 \foreach \x in {-2,-1,0,1,2}
   \draw[farey geodesic] (\x,0)--(\x,2.25);
 \foreach \n in {-2,-1,0,1}{
   \draw[farey geodesic]
     (\n,0)..controls +(0.18,1.15) and +(-0.18,1.15)..(\n+1,0);
   \draw[farey geodesic]
     (\n,0)..controls +(0.10,0.50) and +(-0.10,0.50)..(\n+0.5,0);
   \draw[farey geodesic]
     (\n+0.5,0)..controls +(0.10,0.50) and +(-0.10,0.50)..(\n+1,0);
 }

 \coordinate (bm2) at (-1.5,2.08);
 \coordinate (wm1) at (-1,2.38);
 \coordinate (bm1) at (-0.5,2.08);
 \coordinate (wzero) at (0,2.38);
 \coordinate (bzero) at (0.5,2.08);
 \coordinate (wone) at (1,2.38);
 \coordinate (bone) at (1.5,2.08);
 \coordinate (wtwo) at (2,2.38);
 \coordinate (btwo) at (2.28,2.18);
 \coordinate (dm1) at (-0.5,0.86);
 \coordinate (dzero) at (0.5,0.86);
 \coordinate (done) at (1.5,0.86);
 \coordinate (cm1) at (-0.5,0.57);
 \coordinate (czero) at (0.5,0.57);
 \coordinate (cone) at (1.5,0.57);
 \coordinate (qonezerozero) at (-0.75,0.37);
 \coordinate (qoneoneone) at (-0.25,0.37);
 \coordinate (qzerozerozero) at (0.25,0.37);
 \coordinate (qzerozeroone) at (0.75,0.37);
 \coordinate (ronezerozero) at (-0.90,0.17);
 \coordinate (roneoneone) at (-0.10,0.17);
 \coordinate (rzerozerozero) at (0.10,0.17);
 \coordinate (rzerozeroone) at (0.90,0.17);

 \draw[farey edge] (bm2)--(wm1)--(bm1)--(wzero)--(bzero)--(wone)--(bone)--(wtwo)--(btwo);
 \draw[farey edge] (bm1)--(dm1)--(cm1);
 \draw[farey edge] (bzero)--(dzero)--(czero);
 \draw[farey edge] (bone)--(done)--(cone);
 \draw[farey edge] (cm1)--(qonezerozero)--(ronezerozero);
 \draw[farey edge] (cm1)--(qoneoneone)--(roneoneone);
 \draw[farey edge] (czero)--(qzerozerozero)--(rzerozerozero);
 \draw[farey edge] (czero)--(qzerozeroone)--(rzerozeroone);

 \node[word label,above=1pt] at ($(wzero)!0.50!(bzero)$) {$I$};
\node[code label,below=1pt, font=\normalsize] at ($(wzero)!0.50!(bzero)$) {$0$};
 \node[word label,above=1pt] at ($(wzero)!0.50!(bm1)$) {$S$};
 \node[code label,below=1pt, font=\normalsize] at ($(wzero)!0.50!(bm1)$) {$1$};
 \node[word label,above=2pt] at ($(bm1)!0.50!(wm1)$) {$SL^2$};
 \node[word label,left=2pt] at ($(bm1)!0.43!(dm1)$) {$SL$};
 \node[word label,left=2pt] at ($(dm1)!0.50!(cm1)$) {$SLS$};
 \node[word label,right=2pt] at ($(bzero)!0.43!(dzero)$) {$L^2$};
 \node[word label,right=2pt] at ($(dzero)!0.50!(czero)$) {$L^2S$};
 \node[word label,above=2pt] at ($(bzero)!0.50!(wone)$) {$L$};

 \node[code label,below=3pt, font=\normalsize] at (wm1) {$10$};
 \node[code label,right=4pt, font=\normalsize] at (dm1) {$11$};
 \node[code label,left=4pt, font=\normalsize] at (dzero) {$00$};
 \node[code label,below=3pt, font=\normalsize] at (wone) {$01$};

 \node[word label,above=2pt] at ($(wm1)!0.50!(bm2)$) {$SL^2S$};
 \node[word label,above=2pt] at ($(wone)!0.50!(bone)$) {$LS$};
 \node[word label,above=2pt] at ($(bone)!0.50!(wtwo)$) {$LSL$};
 \node[word label,above=2pt] at ($(wtwo)!0.50!(btwo)$) {$LSLS$};
 \node[word label,left=2pt] at ($(bone)!0.43!(done)$) {$LSL^2$};
 \node[word label,left=2pt] at ($(done)!0.50!(cone)$) {$LSL^2S$};
 \node[code label,below=3pt, font=\normalsize] at (wtwo) {$011$};
 \node[code label,right=4pt,font=\normalsize] at (done) {$010$};
 \node[code label,right=1pt, below=4pt, font=\footnotesize] at (qonezerozero) {$110$};
 \node[code label,left=1pt, below=3pt, font=\footnotesize] at (qoneoneone) {$111$};
 \node[code label,right=3pt, below=3pt, font=\footnotesize] at (qzerozerozero) {$000$};
 \node[code label,left=3pt, below=3pt, font=\footnotesize] at (qzerozeroone) {$001$};

 \foreach \p in {wm1,wzero,wone,wtwo,dm1,dzero,done,
                   qonezerozero,qoneoneone,qzerozerozero,qzerozeroone}
   \node[open vertex] at (\p) {};
 \foreach \p in {bm2,bm1,bzero,bone,btwo,cm1,czero,cone,
                   ronezerozero,roneoneone,rzerozerozero,rzerozeroone}
   \node[filled vertex] at (\p) {};

 \foreach \x/\lab in {
   -2/{-2},-1.5/{-\frac32},-1/{-1},-0.5/{-\frac12},
   0/{0},0.5/{\frac12},1/{1},1.5/{\frac32},2/{2}}
   \node[below=3pt,font=\scriptsize] at (\x,0) {$\lab$};
 \end{tikzpicture}
 \caption{The binary coding near $e_I$.  The labels $0$ and $1$ belong to the initial edges $I$ and
 $S$. The remaining binary labels are attached to edge pairs $(W,WS)$, where $W$ is a word ending with $L$ or $L^{2}$.}
 \label{fig:kappa-coding}
\end{figure}
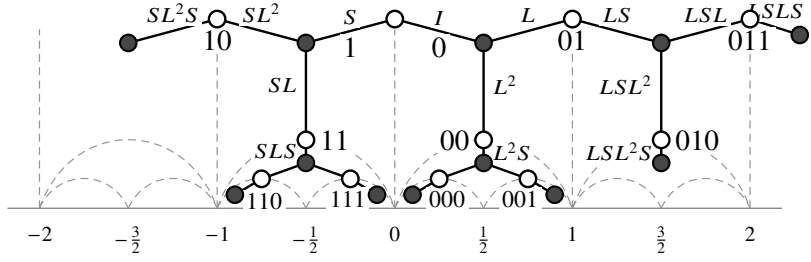

The map $\kappa$ is a homeomorphism. Under this coding,
$\cyl{I}$ and $\cyl{S}$ correspond to $0\mathcal B$ and
$1\mathcal B$, respectively. Finite partitions into at least
two Farey cylinders correspond exactly to finite maximal
binary prefix codes with at least two elements.
The singleton code consisting of the empty binary word
corresponds to the trivial partition $\{\partial_I\F\}$.
More generally, if admissible Farey prefixes $U,V$ have
binary codes $u,v$, then conjugation by $\kappa$ sends the prefix
replacement $UY\mapsto VY$ to
\begin{equation}\label{eq:farey-prefix-conjugacy}
 u\xi\longmapsto v\xi.
\end{equation}

Let $\sim_{\mathrm{bin}}$ be the equivalence relation on $\mathcal B$
generated by
\[
 w0\,1^\infty\sim_{\mathrm{bin}}w1\,0^\infty
 \quad(w\in\{0,1\}^{*}),
 \qquad
 0^\infty\sim_{\mathrm{bin}}1^\infty,
\]
and put $\mathcal B_{\mathrm{circ}}:=\mathcal B/\!\sim_{\mathrm{bin}}$.
This is the binary circle.  Under $\kappa$, the pairs in
$\sim_{\mathrm{bin}}$ are exactly the pairs identified in the smashed
boundary.  Hence $\kappa$ descends to a homeomorphism
\begin{equation}\label{eq:kappa-circle-homeo}
 \bar\kappa:\partial_I^\#\F
 \xrightarrow{\ \cong\ }\mathcal B_{\mathrm{circ}}.
\end{equation}

Under $\kappa$, the linear-order, cyclic-order, and arbitrary prefix
replacements are precisely the actions of $\mathsf F$, $\mathsf T$, and
$\mathsf V$ defined in Section~\ref{sec:thompson-prefix}.  Cyclic-order-preserving prefix replacements preserve $\sim_{\mathrm{bin}}$, so
$\mathsf T$ acts on $\mathcal B_{\mathrm{circ}}$.  Flips descend to this
circle quotient; those fixing the point represented by $0^\infty$ and $1^\infty$ lie in $\mathsf F$.

\begin{remark}[Question-mark conjugacy] In fact, the composition
 \[
\bar\kappa\circ\overline{\ev}^{-1}:  \PP^1(\R)\longrightarrow\mathcal B_{\mathrm{circ}}
 \]
 is the circular Minkowski question-mark function; it sends rational points
 to dyadic points.
\end{remark}

\subsection{Thompson group identifications and modular generators}

We use the following standard prefix generators, which may be found in \cite[\S\S4--5]{CannonFloydParry1996} and \cite[p.~1112]{CannonFloyd2011}.
\begin{align*}
	 X_0 : & \qquad
	0\xi\mapsto00\xi,\qquad
	10\xi\mapsto01\xi,\qquad
	11\xi\mapsto1\xi, \\
	X_1: & \qquad
	0\xi\mapsto0\xi,\qquad
	10\xi\mapsto100\xi,\qquad
	110\xi\mapsto101\xi,\qquad
	111\xi\mapsto11\xi,\\
	 Y: & \qquad 0\xi\mapsto1\xi,\qquad1\xi\mapsto0\xi.
\end{align*}
Then
\begin{equation}\label{eq:FT-generators}
 \mathsf F=\langle X_0,X_1\rangle,\qquad
 \mathsf T=\langle X_0,X_1,Y\rangle,
\end{equation}
as shown in \cite[\S3--5]{CannonFloydParry1996}.  

\begin{theorem}\label{thm:flip-conjugacy}
Let
\[
	G_{\mathrm{flip}} :=  \left\langle \phi_I,\phi_{WL},\phi_{WL^2} : W\text{ admissible}  \right\rangle.
\]
The group $G_{\mathrm{flip}}$ is isomorphic to Thompson's group $\mathsf T$ in its prefix-code realization. More precisely, 
\begin{equation}\label{eq:flip-generators}
 X_0=\phi_I\phi_{SL^2}\phi_I^3\phi_{SL^2}\phi_I^3,
 \qquad X_1=\phi_{SL},\qquad Y=\phi_I^2.
\end{equation}
\end{theorem}

\begin{proof}
	The flip $\phi_I$ acts by
	\[
	 00 \xi \mapsto01\xi\mapsto 10\xi\mapsto 11\xi \mapsto00\xi.
	\]
	Taking $W=I$ in \cref{def:left-flip} gives
	\begin{equation}\label{eq:phi-L-binary}
		\phi_L:
		\quad
		00\xi\mapsto000\xi,\qquad
		010\xi\mapsto001\xi,\qquad
		011\xi\mapsto01\xi,
		\qquad 1\xi\mapsto1\xi.
	\end{equation}
	Thus $\phi_L$ is supported on $0\mathcal B$.  Taking $W=S$ gives
	\begin{equation}\label{eq:phi-SL-binary}
		\phi_{SL}:
		\quad
		0\xi\mapsto0\xi,\qquad
		10\xi\mapsto100\xi,\qquad
		110\xi\mapsto101\xi,\qquad
		111\xi\mapsto11\xi.
	\end{equation}
	Its inverse is the right flip
	\begin{equation}\label{eq:phi-SL2-binary}
		\phi_{SL^2}:
		\quad
		0\xi\mapsto0\xi,\qquad
		100\xi\mapsto10\xi,\qquad
		101\xi\mapsto110\xi,\qquad
		11\xi\mapsto111\xi.
	\end{equation}
	More generally, every flip has a cyclic-order-preserving prefix table, so we deduce $G_{\mathrm{flip}}\leq\mathsf T$.  For the reverse
	inclusion, we calculate the standard generators in flip coordinates. Indeed, $X_1=\phi_{SL}$, $\phi_I^2$ exchanges $0\xi$ and $1\xi$, and hence
	$Y=\phi_I^2$.  Finally, composing this four-cycle with the right-flip table \eqref{eq:phi-SL2-binary}, from right to left, gives
	\[
		 \phi_I\phi_{SL^2}\phi_I^3\phi_{SL^2}\phi_I^3:
		 \quad
		 0\xi\mapsto00\xi,\qquad
		 10\xi\mapsto01\xi,\qquad
		 11\xi\mapsto1\xi.
	\]
	This is the defining table of $X_0$.  Thus the generators in \eqref{eq:FT-generators} belong to $G_{\mathrm{flip}}$, proving the reverse
	inclusion. 
\end{proof}


\paragraph{Standing convention.}
Using \cref{thm:flip-conjugacy} and
\eqref{eq:kappa-circle-homeo}, we identify $\mathsf T$
with $G_{\mathrm{flip}}$ and write $t$ also for its
transported action $\bar\kappa^{-1}t\bar\kappa$
on $\partial_I^\#\F$. Thus $t(x)$, and later $t(\mathscr E)$ and $t(\vec e)$ for the patterns and marked edges of \cref{sec:effective-decompositions}, refer to this boundary action.


\begin{example}
	The modular generator $S$ exchanges the initial binary
	cylinders, so $S=\phi_I^2$. The generator $L$ has table
	\begin{equation}\label{eq:L-binary-table}
		L:\quad
		10\xi\mapsto010\xi,\qquad
		11\xi\mapsto011\xi,\qquad
		00\xi\mapsto1\xi,\qquad
		01\xi\mapsto00\xi.
	\end{equation}
	Direct composition using \eqref{eq:base-flip} and
	\eqref{eq:phi-SL-binary} gives
	\begin{equation}\label{eq:modular-generators-as-flips}
		\begin{aligned}
			S&=\phi_I^2,\\
			L&=\phi_I^3\phi_{SL}\phi_I\phi_{SL}\phi_I^3,\\
			L^2&=\phi_I\phi_{SL^2}\phi_I^3\phi_{SL^2}\phi_I.
		\end{aligned}
	\end{equation}
\end{example}

Define the flip--shuffle group on the ordinary boundary by
\[
 G_{\mathrm{flip,shuffle}}
 :=\left\langle
     G_{\mathrm{flip}},\ \sigma_W:
     W\text{ admissible}
   \right\rangle.
\]
%
%

\begin{corollary}[Flip--shuffle realization]
	\label{thm:flip-shuffle-V}
	\begin{equation}\label{eq:V-generators}
		G_{\mathrm{flip,shuffle}}\cong\mathsf V.
	\end{equation}
\end{corollary}

\begin{proof}
	Under $\kappa$, every flip and shuffle is a finite
	prefix replacement, giving the inclusion in $\mathsf V$.
	Conversely, $\mathsf V=\langle\mathsf T,Z\rangle$,
	where $Z$ exchanges $10\mathcal B$ and $11\mathcal B$
	by prefix replacement and fixes $0\mathcal B$;
	see \cite[\S\S4--6]{CannonFloydParry1996}.
	By \cref{def:shuffle}, $Z=\sigma_S$, while
	\cref{thm:flip-conjugacy} identifies $\mathsf T$
	with $G_{\mathrm{flip}}$. This proves the reverse
	inclusion.
\end{proof}

A further corollary of \cref{thm:flip-conjugacy} is that the group generated
by $\phi_W$ for prefixes $W$ starting with $L^2S$ is isomorphic to
Thompson's group $\mathsf F$.

\subsection{Relations among flips and shuffles}
\label{sec:flip-shuffle-relations}

We now collect the local relations among flips and shuffles before reducing
them to a smaller generating set.  In what follows, $U$ and $V$ are admissible
prefixes, $R$ is an arbitrary finite word in $\{A,B\}$, and products are
composed from right to left.  Parentheses in a subscript separate a supporting
prefix from the final $L$ or $L^2$ specifying a left or right flip.

\begin{proposition}[Local relation families]
	\label{prop:flip-shuffle-relations}
	The following identities hold among the boundary flips and shuffles.
	
	\smallskip
	\noindent
	\emph{Flip--flip relations.}  A right flip is the inverse of the
	corresponding left flip:
	\begin{equation}\label{eq:flip-inverse}
		\phi_{UL^2}=\phi_{UL}^{-1}.
	\end{equation}
	If $\cyl{U}\cap\cyl{V}=\varnothing$, then
	\begin{equation}\label{eq:flip-disjoint-commutation}
		[\phi_{UL},\phi_{VL}]=1.
	\end{equation}
	For nested cylinders one has
	\begin{align}
		\phi_{UL}\phi_{(UBR)L}\phi_{UL^2}
		&=\phi_{(UBBR)L}, \label{eq:flip-transport-B}\\
		\phi_{UL}\phi_{(UABR)L}\phi_{UL^2}
		&=\phi_{(UBAR)L}, \label{eq:flip-transport-AB}\\
		\phi_{UL}\phi_{(UAAR)L}\phi_{UL^2}
		&=\phi_{(UAR)L}. \label{eq:flip-transport-AA}
	\end{align}
	Finally, the pentagon, or splitting, relation is
	\begin{equation}\label{eq:flip-pentagon}
		\phi_{UL}^{\,2}
		=\phi_{(UB)L}\phi_{UL}\phi_{(UA)L}.
	\end{equation}
	
	\smallskip
	\noindent
	\emph{Shuffle--shuffle relations.}  Every shuffle is an involution:
	\begin{equation}\label{eq:shuffle-involution}
		\sigma_U^2=1.
	\end{equation}
	If $\cyl{U}\cap\cyl{V}=\varnothing$, then
	\begin{equation}\label{eq:shuffle-disjoint-commutation}
		[\sigma_U,\sigma_V]=1,
	\end{equation}
	whereas for nested cylinders one has
	\begin{equation}\label{eq:shuffle-transport}
		\sigma_U\sigma_{UBR}\sigma_U=\sigma_{UAR}.
	\end{equation}
	
	\smallskip
	\noindent
	\emph{Mixed relations.}  If $\cyl{U}\cap\cyl{V}=\varnothing$, then
	\begin{equation}\label{eq:mixed-disjoint-commutation}
		[\phi_{UL},\sigma_V]=1.
	\end{equation}
	Conjugation of a shuffle by a flip gives
	\begin{align}
		\phi_{UL}\sigma_{UBR}\phi_{UL^2}
		&=\sigma_{UBBR}, \label{eq:mixed-flip-transport-B}\\
		\phi_{UL}\sigma_{UABR}\phi_{UL^2}
		&=\sigma_{UBAR}, \label{eq:mixed-flip-transport-AB}\\
		\phi_{UL}\sigma_{UAAR}\phi_{UL^2}
		&=\sigma_{UAR}. \label{eq:mixed-flip-transport-AA}
	\end{align}
	Conversely, conjugation by a shuffle interchanges the two descendant
	cylinders:
	\begin{equation}\label{eq:mixed-shuffle-transport}
		\sigma_U\phi_{(UBR)L}\sigma_U=\phi_{(UAR)L}.
	\end{equation}
	The remaining same-support interaction is the hexagon relation
	\begin{equation}\label{eq:flip-shuffle-hexagon}
		\phi_{UL}\sigma_U\phi_{UL}
		=\sigma_{UB}\phi_{UL}\sigma_{UA}.
	\end{equation}
\end{proposition}

\begin{proof}
	All the relations are verified directly from their definitions. Disjointly supported maps commute.  The
	flip $\phi_{UL}$ carries the supporting prefixes
	\[
	UBR,\qquad UABR,\qquad UAAR
	\]
	respectively to
	\[
	UBBR,\qquad UBAR,\qquad UAR,
	\]
	which proves both the flip transport relations
	\eqref{eq:flip-transport-B}--\eqref{eq:flip-transport-AA} and their mixed
	counterparts
	\eqref{eq:mixed-flip-transport-B}--\eqref{eq:mixed-flip-transport-AA}.
	The shuffle $\sigma_U$ interchanges the prefixes $UBR$ and $UAR$, giving
	\eqref{eq:shuffle-transport} and
	\eqref{eq:mixed-shuffle-transport}.  Finally, refining the relevant prefix
	codes and comparing the substitutions on each cylinder gives
	\eqref{eq:flip-pentagon} and \eqref{eq:flip-shuffle-hexagon}.  Thus every
	displayed identity follows on $\cyl{U}$, and both sides are the identity
	outside that cylinder.
\end{proof}


\section{Effective decompositions into flips and shuffles}
\label{sec:effective-decompositions}

The preceding identifications are structural.  We next make them effective
by giving finite procedures that convert prefix data into words in flips and
shuffles.

\paragraph{Computational convention.}
The input is an explicit prefix-code table.  Binary prefixes
and $A$--$B$ addresses are stored as words, and reading,
writing, or comparing them is charged according to their
length.  When integer coordinates are used, arithmetic is
charged at bit cost with schoolbook algorithms.
All complexity bounds refer to the explicit prefix-table
size defined below; they do not assert polynomial complexity
in a compressed continued-fraction or rational-coordinate
encoding alone.

\subsection{Elements of \texorpdfstring{$\T$}{T} as sequences of flips}
For a cyclic prefix table
\[
 t=h_{P,Q,\tau},\qquad
 P=\{p_1,\ldots,p_n\},\quad Q=\{q_1,\ldots,q_n\},
\]
define its \emph{input prefix-table size} by
\begin{equation}\label{eq:prefix-table-size}
 s(P,Q):=n+\sum_{i=1}^n\bigl(|p_i|+|q_i|\bigr).
\end{equation}

Let $\mathscr R\subset \partial_I^{\#}\F$ be the set of classes of eventually
constant ends.  A \emph{boundary pair} is an unordered pair
$\langle x,y\rangle$ of distinct elements of $\mathscr R$.  Two boundary
pairs $\langle x,y\rangle$ and $\langle u,v\rangle$ are \emph{linked}, written
\[
 \langle x,y\rangle\pitchfork\langle u,v\rangle,
\]
if their endpoints alternate in the cyclic order of
$\partial_I^{\#}\F$; equivalently,
after interchanging $u$ and $v$ if necessary,
$u\in(x,y)$ and $v\in(y,x)$.  Define the standard Farey boundary pattern by
\begin{equation}\label{eq:standard-boundary-pattern}
 \mathscr E_0:=
 \left\{
  \left\langle\frac ac,\frac bd\right\rangle:
  |ad-bc|=1
 \right\}.
\end{equation}
Here fractions are reduced and interpreted projectively.  The family $\mathscr E_0$ is pairwise unlinked and maximal with
respect to inclusion: every boundary pair not belonging to
$\mathscr E_0$ is linked with some member of $\mathscr E_0$.
The modular group acts simply transitively on the oriented  boundary pairs of
$\mathscr E_0$.  Hence each orientation of a member of $\mathscr E_0$ has a
unique label $g\in\psl$ carrying $(\infty,0)$ to that particular orientation.  Fix the
shortlex order on the reduced words in the free-product generators
$S,L,L^2$, with $S<L<L^2$.  For each unordered boundary pair choose the orientation
whose label has the smaller reduced word, and order the boundary pairs by
these chosen labels.  This gives an effective enumeration
\begin{equation}\label{eq:farey-edge-order}
 \alpha_1\prec\alpha_2\prec\alpha_3\prec\cdots
 \qquad(\alpha_j\in\mathscr E_0).
\end{equation}

We call a maximal pairwise unlinked family $\mathscr E$ a
\emph{Farey-type boundary pattern} if the symmetric difference
$\mathscr E\mathbin\triangle\mathscr E_0$ is finite.  Such a pattern has the
same local incidence property as $\mathscr E_0$.  More precisely, if
$e=\langle u,v\rangle\in\mathscr E$, a \emph{common neighbour} of $u$ and
$v$ in $\mathscr E$ is a point $w\in\mathscr R\setminus\{u,v\}$ such that
\[
 \langle u,w\rangle,\ \langle v,w\rangle\in\mathscr E.
\]
The endpoints of every $e\in\mathscr E$ have precisely two common
neighbours, one in each of the two complementary intervals determined by
$u$ and $v$. To see this, choose a finite partition into Farey cylinders
sufficiently fine that every member of
$\mathscr E_0\setminus\mathscr E$ joins nonconsecutive
partition endpoints. The pairs joining consecutive endpoints
then belong to $\mathscr E_0\cap\mathscr E$.
Every member of $\mathscr E\setminus\mathscr E_0$ is linked
with a deleted standard pair and unlinked with these common
boundary pairs; hence it also joins nonconsecutive partition
endpoints. Thus all changes are confined to this finite
cyclic list.

The restriction of $\mathscr E$ to this list is maximal
pairwise unlinked. Induction, splitting the list along a
present nonconsecutive pair, shows that each consecutive
pair has exactly one common neighbour among the remaining
points of the list, while each present nonconsecutive pair
has two. The unchanged Farey subdivisions in the
complementary cylinder arcs supply the other neighbour
of each consecutive pair. All other pairs retain their
standard local incidence. This proves the assertion.

Fix the oriented standard Farey edge
\[
 \vec e_0:=(\infty,0).
\]
A \emph{marked Farey-type boundary pattern} is a pair
$(\mathscr E,\vec e)$, where $\mathscr E$ is a Farey-type boundary pattern
and $\vec e=(x,y)$ is an orientation of a member
$\langle x,y\rangle\in\mathscr E$.  

For $t\in\mathsf T$ set
\[
 t(\mathscr E)
 :=\{\langle t(u),t(v)\rangle:
       \langle u,v\rangle\in\mathscr E\},
 \qquad
 t(\vec e):=(t(x),t(y)).
\]
Thus $t$ sends every boundary pair to the pair formed by the images of its
endpoints and transports the orientation of the distinguished edge in the
same order.  Since elements of $\mathsf T$ preserve cyclic order, they
preserve the unlinking relation.  

By \cref{lem:effective-finite-support} below,
$t(\mathscr E_0)\mathbin\triangle\mathscr E_0$ is finite.
Thus
\[
(\mathscr E_t,\vec e_t)
:=\bigl(t(\mathscr E_0),t(\vec e_0)\bigr)
\]
is a marked Farey-type boundary pattern. More generally,
\[
t(\mathscr E)\mathbin\triangle t(\mathscr E_0)
=t(\mathscr E\mathbin\triangle\mathscr E_0)
\]
is finite for every Farey-type pattern $\mathscr E$.
Together with cyclic-order preservation, this shows that
\[
t\cdot(\mathscr E,\vec e)
:=\bigl(t(\mathscr E),t(\vec e)\bigr)
\]
defines an action on marked Farey-type boundary patterns.
The stabilizer of $\mathscr E_0$ in $\mathsf T$ is exactly
$\psl$. Indeed, let $h$ preserve $\mathscr E_0$, and choose
$m\in\psl$ with $m(\vec e_0)=h(\vec e_0)$.
Then $m^{-1}h$ preserves $\mathscr E_0$ and fixes $0$
and $\infty$. Since it preserves cyclic order, it fixes
their unique common neighbour in each complementary
interval. Repeating this argument through the Farey
subdivisions shows that it fixes every rational boundary
point. Density and continuity therefore give $m^{-1}h=1$. 
Consequently, the stabilizer of
$(\mathscr E_0,\vec e_0)$ is trivial.
For $t\in\mathsf T$, the finite symmetric difference
\[
\operatorname{Def}(t):=
\mathscr E_t\mathbin\triangle\mathscr E_0
\]
is called the \emph{boundary defect} of $t$. For an admissible prefix $W$, put $\alpha_W:=\langle W(0),W(\infty)\rangle\in\mathscr E_0$.

\begin{lemma}
\label{lem:effective-finite-support}
From a cyclic prefix table for $t\in\mathsf T$, one can compute the marked
edge $\vec e_t$ and the two finite sets
\[
 \mathscr E_t\setminus\mathscr E_0,
 \qquad
 \mathscr E_0\setminus\mathscr E_t,
\]
and hence $\operatorname{Def}(t)$.  If the table has $n$ entries and size
$s=s(P,Q)$, then
\begin{equation}\label{eq:defect-table-bound}
 \bigl|\mathscr E_t\setminus\mathscr E_0\bigr|\leq n-1,
 \qquad
 \bigl|\mathscr E_0\setminus\mathscr E_t\bigr|\leq n-1,
 \qquad
 |\operatorname{Def}(t)|\leq2(n-1)\leq2s.
\end{equation}
\end{lemma}

\begin{proof}
	Write $t=h_{P,Q,\tau}$. If the table has one row, then $t=1$,
	so the defect is empty and $\vec e_t=\vec e_0$.
	Assume henceforth that $n\geq2$.
	
	For $R=P$ or $R=Q$, let
	\[
	\operatorname{Int}^*(R)
	:=
	\{u\in\{0,1\}^{*}:
	u\ne\varnothing,\ 
	u\text{ is a proper prefix of some }r\in R\}.
	\]
	For each such $u$, let $\widehat u$ be its admissible Farey
	prefix, and put
	\[
	\mathscr A(R)
	:=
	\{\alpha_I\}\cup
	\{\alpha_{\widehat u}:u\in\operatorname{Int}^*(R)\}.
	\]
	A maximal prefix code with $n$ words is obtained from
	$\{\varnothing\}$ by $n-1$ cylinder splittings.  Exactly one
	splitting occurs at the empty prefix, so
	\[
	|\operatorname{Int}^*(R)|=n-2,
	\qquad
	|\mathscr A(R)|\leq n-1.
	\]
	
	Every standard boundary pair outside $\mathscr A(P)$ lies in
	the closed Farey arc associated with a word of $P$.
	On that arc, $t$ agrees with an element of $\psl$, by
	\eqref{eq:prefix-projective}, and hence maps the pair to a
	standard boundary pair. Consequently,
	\[
	\mathscr E_t\setminus\mathscr E_0
	=
	\{t(\alpha):\alpha\in\mathscr A(P),\
	t(\alpha)\notin\mathscr E_0\}.
	\]
	Applying the same argument to
	$t^{-1}=h_{Q,P,\tau^{-1}}$ gives
	\[
	\mathscr E_0\setminus\mathscr E_t
	=
	\{\beta\in\mathscr A(Q):
	t^{-1}(\beta)\notin\mathscr E_0\}.
	\]
	These finite sets are computable: evaluate their endpoints
	using the prefix tables and continued-fraction coordinates,
	and test the Farey condition on primitive integer
	representatives. Their cardinalities are each at most $n-1$,
	which proves \eqref{eq:defect-table-bound}.
	
	Finally, the table computes $t(\infty)$ and $t(0)$, and hence
	the marked edge
	\[
	\vec e_t=(t(\infty),t(0)).
	\]

\end{proof}

We now isolate the finite part on which two Farey-type boundary patterns
differ.  Let $\mathscr E$ and $\mathscr F$ be Farey-type boundary patterns
and put
\[
D_{\mathscr E}:=\mathscr E\setminus\mathscr F,
\qquad
D_{\mathscr F}:=\mathscr F\setminus\mathscr E.
\]
Both sets are finite.

An \emph{$\mathscr E$-boundary triple} is a set of boundary pairs of the
form
\[
\Delta_{\mathscr E}(u,v;w)
:=
\bigl\{
\langle u,v\rangle,
\langle u,w\rangle,
\langle v,w\rangle
\bigr\},
\]
where $w$ is a common neighbour of the endpoints of
$\langle u,v\rangle$ in $\mathscr E$.  If
$e=\langle u,v\rangle\in D_{\mathscr E}$ and $x_e,z_e$ are its two common
neighbours in $\mathscr E$, then $e$ belongs to precisely the two
$\mathscr E$-boundary triples
\[
\Delta_{\mathscr E}(u,v;x_e),
\qquad
\Delta_{\mathscr E}(u,v;z_e).
\]

Let $\Gamma_{\mathscr E,\mathscr F}$ have as vertices the
$\mathscr E$-boundary triples incident to at least one pair in
$D_{\mathscr E}$. For each $e\in D_{\mathscr E}$, join its two
adjacent triples by one edge labelled $e$.
Thus its edge set is naturally identified with $D_{\mathscr E}$.
The graph is finite and has no isolated vertices.  It is
acyclic: an edge labelled $e$ joins boundary triples on opposite
sides of $e$, whereas every other adjacency remains on one
side, since all pairs in $\mathscr E$ are unlinked with $e$.
Deleting that edge therefore separates its endpoints.
Thus every edge is a bridge, and
$\Gamma_{\mathscr E,\mathscr F}$ is a forest.

For a connected component $K$ of $\Gamma_{\mathscr E,\mathscr F}$,
define its boundary by
\[
\partial K
:=
\left\{
\gamma\in\mathscr E:
\gamma\text{ occurs in exactly one boundary triple belonging to }K
\right\}.
\]
Let $k$ be the number of edges of $K$. Since there are no isolated
vertices, $k\geq1$. Since $K$ is a tree, it has $k+1$ vertices,
that is, $k+1$ boundary triples. Gluing these triples along the
$k$ pairs labelling the edges of $K$ leaves
\[
3(k+1)-2k=k+3
\]
boundary pairs. Successively removing a leaf triple shows
inductively that these pairs form a single cyclic list.
Consequently, there are $r=k+3\geq4$ distinct boundary points
$v_1,\ldots,v_r$, listed in cyclic order, such that
\[
\partial K
=
\{\langle v_i,v_{i+1}\rangle:1\leq i\leq r\},
\]
where indices are read modulo $r$.
The cyclic listing $C_K=(v_1,\ldots,v_r)$, understood up to
cyclic shift and reversal, is called the
\emph{disagreement block associated with $K$}.

Moreover, $\partial K\subseteq\mathscr E\cap\mathscr F$.
Indeed, if $\gamma\in\partial K$ did not belong to $\mathscr F$,
then $\gamma\in D_{\mathscr E}$, so both
$\mathscr E$-boundary triples adjacent to $\gamma$ would be
vertices of $\Gamma_{\mathscr E,\mathscr F}$ joined by the edge
labelled $\gamma$. Since one of these triples belongs to $K$,
the other would also belong to $K$, contradicting the definition
of $\partial K$.
	
Given $C_K$, a boundary pair $\langle v_i,v_j\rangle$ is \emph{internal in $C_K$} if
$v_i$ and $v_j$ are not consecutive in the cyclic list.  Define
\[
\begin{aligned}
\mathscr E[C_K]
&:=\{e\in\mathscr E:e\text{ is internal in }C_K\},\\
\mathscr F[C_K]
&:=\{e\in\mathscr F:e\text{ is internal in }C_K\}.
\end{aligned}
\]

\begin{lemma}
	\label{lem:finite-block-decomposition}
	We have:
	\begin{enumerate}
		\item $\mathscr E[C_K]$ and $\mathscr F[C_K]$ are maximal pairwise
		unlinked families of internal pairs of $C_K$;		
		\item $\mathscr E[C_K]\cap\mathscr F[C_K]=\varnothing$;
		
		\item the finite differences decompose as
		\[
		\mathscr E\setminus\mathscr F
		=
		\bigsqcup_K\mathscr E[C_K],
		\qquad
		\mathscr F\setminus\mathscr E
		=
		\bigsqcup_K\mathscr F[C_K].
		\]
	\end{enumerate}
	
	In particular,
	\begin{equation}\label{eq:defect-cardinalities-equal}
		|\mathscr E\setminus\mathscr F|
		=
		|\mathscr F\setminus\mathscr E|
		=
		\sum_K\bigl(|C_K|-3\bigr).
	\end{equation}
	The decomposition can be computed from the finite symmetric difference and
	the common-neighbour relation.
\end{lemma}

\begin{proof}
	Let $K$ have $k$ edges.  By construction, the internal members of
	$\mathscr E$ in $C_K$ are precisely the labels of the edges of $K$.
	Therefore
	\[
	|\mathscr E[C_K]|=k=|C_K|-3.
	\]
	In particular, the families $\mathscr E[C_K]$ partition
	$\mathscr E\setminus\mathscr F$.
	
	There is no pair belonging to both $\mathscr E[C_K]$ and $\mathscr F[C_K]$: a common internal pair would be an internal member of $\mathscr E$, hence a label of $K$, hence not in $\mathscr F$.
		
	Let $f\in\mathscr F\setminus\mathscr E$.  By maximality of $\mathscr E$,
	the pair $f$ is linked with some $e\in\mathscr E$.  Since $f$ is unlinked
	with every member of $\mathscr F$, the pair $e$ cannot belong to
	$\mathscr F$.  Hence
	\[
	e\in\mathscr E\setminus\mathscr F.
	\]
	Let $K$ be the component containing the edge labelled by $e$.  Every member
	of $\partial K$ belongs to $\mathscr F$, so $f$ is unlinked with
	$\partial K$.  Since $f$ is linked with an internal pair of $C_K$, the
	cyclic order implies that $f$ is itself internal in $C_K$.  Moreover, $f$ is internal in no other block $C_{K'}$: the pair $e$ is not internal in $C_{K'}$, since the internal members of $\mathscr E$ in $C_{K'}$ are the labels of $K'$, and $e$ is unlinked with $\partial K'$; hence $e$ is unlinked with every internal pair of $C_{K'}$.  Since also $\mathscr F[C_K]\cap\mathscr E\subseteq\mathscr E[C_K]\cap\mathscr F[C_K]=\varnothing$, the families $\mathscr F[C_K]$ partition $\mathscr F\setminus\mathscr E$.
	
%
	
	Both restrictions are maximal among pairwise unlinked internal
	pairs of $C_K$.  Indeed, let
	$\mathscr H\in\{\mathscr E,\mathscr F\}$.
	Since $\partial K\subseteq\mathscr H$, a member of $\mathscr H$
	linked with an internal pair of $C_K$ must itself be internal
	in $C_K$.  Hence an additional internal pair unlinked with
	$\mathscr H[C_K]$ would be unlinked with all of $\mathscr H$,
	contradicting maximality.
	
	Finally, a maximal pairwise unlinked family of internal pairs on $r$
	cyclically ordered points has $r-3$ members.  Indeed, an internal pair
	divides the cyclic list into two shorter lists, and induction gives
	\[
	(r_1-3)+(r_2-3)+1=r-3,
	\qquad r_1+r_2=r+2.
	\]
	Consequently,
	\[
	|\mathscr E[C_K]|
	=
	|\mathscr F[C_K]|
	=
	|C_K|-3.
	\]
	Summing over the components proves
	\eqref{eq:defect-cardinalities-equal}.
\end{proof}

If $e=\langle u,v\rangle\in\mathscr E$ has common neighbours $x$ and $z$,
the \emph{diagonal exchange at $e$} is
\begin{equation}\label{eq:boundary-diagonal-exchange}
 D_e(\mathscr E):=
 \bigl(\mathscr E\setminus\{\langle u,v\rangle\}\bigr)
 \cup\{\langle x,z\rangle\}.
\end{equation}
The new pair is unlinked with every member of $\mathscr E \setminus \{e\}$;
hence $D_e(\mathscr E)$ is again a Farey-type boundary pattern.  At the
level of boundary patterns this exchange is involutive. 

\begin{lemma}\label{lem:exchange-flip}
	\begin{enumerate}
		\item If $g\in\mathsf T$, $\mathscr E$ is a Farey-type boundary pattern and
		$e\in\mathscr E$, then $g(\mathscr E)$ is a Farey-type boundary pattern and
		\[
		g\bigl(D_e(\mathscr E)\bigr)=D_{g(e)}\bigl(g(\mathscr E)\bigr).
		\]
		\item Every $e\in\mathscr E_0\setminus\{\alpha_I\}$ equals $\alpha_{UB}$ or $\alpha_{UA}$
		for a unique admissible prefix $U$ and a unique letter.  Put
		\[
		\Phi(e):=
		\begin{cases}
			\phi_I,& e=\alpha_I,\\
			\phi_{UL},& e=\alpha_{UB},\\
			\phi_{UL^2},& e=\alpha_{UA}.
		\end{cases}
		\]
		Then $\Phi(e)(\mathscr E_0)=D_e(\mathscr E_0)$ for every $e\in\mathscr E_0$.
	\end{enumerate}
\end{lemma}

\begin{proof}
	(1) The map $g$ is a bijection of $\mathscr R$ preserving cyclic order, and so
	is $g^{-1}$.  Hence $g$ preserves linking, maximality and common neighbours,
	and
	$g(\mathscr E)\mathbin\triangle\mathscr E_0\subseteq
	g(\mathscr E\mathbin\triangle\mathscr E_0)\cup
	\bigl(g(\mathscr E_0)\mathbin\triangle\mathscr E_0\bigr)$ is finite.
	
	(2) The group $\psl$ acts transitively on $\mathscr E_0$, and the stabilizer of
	$\alpha_I=\langle0,\infty\rangle$ is $\{1,S\}$.  Every element of $\psl$ is $W$
	or $WS$, where the reduced word of $W$ is empty or ends in $S$; such $W$ are
	exactly the admissible prefixes.  Thus $e=\alpha_W$ for an admissible $W$, which
	is unique unless $W\in\{I,S\}$, that is, unless $e=\alpha_I$.  For
	$e\neq\alpha_I$, deleting the last turn of $W$ gives the unique $U$ and letter
	with $W\in\{UB,UA\}$.
	
	By \eqref{eq:prefix-projective}, $\phi_L$ maps the closed arcs $[0,1]$, $[1,2]$,
	$[2,\infty]$ onto $[0,\frac12]$, $[\frac12,1]$, $[1,\infty]$ by elements of
	$\psl$, and fixes $[\infty,0]$ pointwise.  A standard pair is unlinked with the
	standard pairs bounding these arcs.  Hence it lies in one arc, and is then
	mapped to a standard pair, or it joins two of the breakpoints $0,1,2,\infty$.
	The only standard pair of the second kind bounding no arc is
	$\langle1,\infty\rangle$, and
	$\phi_L\langle1,\infty\rangle=\langle\frac12,\infty\rangle\notin\mathscr E_0$.
	The same argument for $\phi_L^{-1}$, with breakpoints $0,\frac12,1,\infty$,
	shows that $\langle0,1\rangle$ is the only standard pair whose preimage,
	$\langle0,2\rangle$, is not standard.  Since $\frac12$ and $\infty$ are the
	common neighbours of $0$ and $1$, and $0$ and $2$ are those of $1$ and $\infty$,
	\[
	\phi_L(\mathscr E_0)=D_{\alpha_B}(\mathscr E_0),
	\qquad
	\phi_{L^2}(\mathscr E_0)=\phi_L^{-1}(\mathscr E_0)=D_{\alpha_A}(\mathscr E_0).
	\]
	By \eqref{eq:base-flip}, $\phi_I$ permutes the arcs $[0,1]$, $[1,\infty]$,
	$[\infty,-1]$, $[-1,0]$ cyclically by elements of $\psl$.  The only standard
	pair joining two breakpoints and bounding no arc is $\langle0,\infty\rangle$,
	and $\phi_I^{\pm1}\langle0,\infty\rangle=\langle-1,1\rangle\notin\mathscr E_0$.
	Hence $\phi_I(\mathscr E_0)=D_{\alpha_I}(\mathscr E_0)$.
	
	Finally, let $U$ be admissible.  Left multiplication by $U$ maps $\cyl{I}$ onto
	$\cyl{U}$ by $Y\mapsto UY$, so \cref{def:left-flip} gives
	$\phi_{UL}=U\phi_LU^{-1}$ and $\phi_{UL^2}=U\phi_{L^2}U^{-1}$.  Since
	$U(\mathscr E_0)=\mathscr E_0$, $U(\alpha_B)=\alpha_{UB}$ and
	$U(\alpha_A)=\alpha_{UA}$, part (1) gives
	\[
	\phi_{UL}(\mathscr E_0)=U\bigl(D_{\alpha_B}(\mathscr E_0)\bigr)
	=D_{\alpha_{UB}}(\mathscr E_0),
	\qquad
	\phi_{UL^2}(\mathscr E_0)=D_{\alpha_{UA}}(\mathscr E_0).
	\qedhere
	\]
\end{proof}

For a disagreement block $C=(v_1,\ldots,v_r)$, write
\[
\partial C
:=\{\langle v_i,v_{i+1}\rangle:1\leq i\leq r\},
\qquad v_{r+1}:=v_1.
\]
Thus $\partial C_K=\partial K$. Define the \emph{fan in $C$ based
at $v_1$} by
\begin{equation}\label{eq:boundary-fan}
 \mathscr F_C(v_1):=
 \partial C\cup
 \{\langle v_1,v_j\rangle:3\leq j\leq r-1\}.
\end{equation}

\paragraph{Fan reduction in a block.}
%

Let $\mathscr M$ be a maximal pairwise unlinked family of internal pairs of
$C=(v_1,\ldots,v_r)$.  Put
\[
d_{\mathscr M}(v_1)
:=
\bigl|
\{\langle v_1,v_j\rangle\in\mathscr M:3\leq j\leq r-1\}
\bigr|.
\]

\begin{lemma}[Block fan reduction]
	\label{lem:block-fan-reduction}
	Let $C=(v_1,\ldots,v_r)$, with $r\geq4$, and let $\mathscr M$ be a maximal
	pairwise unlinked family of internal pairs of $C$.  There is a sequence of
	\[
	r-3-d_{\mathscr M}(v_1)
	\]
	diagonal exchanges, all supported inside $C$ and fixing $\partial C$, which
	transforms
	\[
	\mathscr M\cup\partial C
	\]
	into the fan $\mathscr F_C(v_1)$.  In particular, the sequence contains at
	most $r-3$ exchanges.
\end{lemma}
\begin{proof}
	List in cyclic order the points joined to $v_1$ by members of
	$\mathscr M\cup\partial C$.  Suppose that this list does not contain every
	point of $C$ other than $v_1$.  Choose two consecutive members $u,v$ for
	which the interval from $u$ to $v$ not containing $v_1$ contains a point of
	$C$.
	
	The pair $\langle u,v\rangle$ is unlinked with every member of
	$\mathscr M$.  Indeed, a pair linked with $\langle u,v\rangle$ would either
	be linked with one of $\langle v_1,u\rangle$ and
	$\langle v_1,v\rangle$, or would join $v_1$ to a point strictly between
	$u$ and $v$, contradicting the choice of consecutive neighbours.
	Maximality therefore gives
	\[
	\langle u,v\rangle\in\mathscr M.
	\]
	
	One common neighbour of $u$ and $v$ is $v_1$.  Let $z$ be the other.  Then
	$z$ lies in the interval from $u$ to $v$ not containing $v_1$, and
	$\langle v_1,z\rangle$ is not present.  Hence the diagonal exchange
	\[
	\langle u,v\rangle\longmapsto\langle v_1,z\rangle
	\]
	is defined.  It adds one internal pair incident with $v_1$ and removes none
	of the pairs already incident with $v_1$.
	
	Each repetition increases by one the number of internal pairs incident with $v_1$.  The fan contains exactly $r-3$ internal pairs incident with
	$v_1$, so the process terminates after
	\[
	r-3-d_{\mathscr M}(v_1)
	\]
	exchanges.  At termination every nonconsecutive vertex is joined to $v_1$,
	and the resulting family is therefore $\mathscr F_C(v_1)$.
\end{proof}

For $t\in\mathsf T$, put
\begin{equation}\label{eq:boundary-pair-defects}
 k_-(t):=\bigl|\mathscr E_0\setminus\mathscr E_t\bigr|,
 \qquad
 k_+(t):=\bigl|\mathscr E_t\setminus\mathscr E_0\bigr|.
\end{equation}
Let $N_{\min}(t)$ be the least number of elementary diagonal exchanges in
any sequence carrying the unmarked pattern $\mathscr E_t$ to
$\mathscr E_0$.  For the deterministic reduction below, let $N(t)$ denote
the number of exchanges in its fan stage; it does not include the flips used
to express the residual modular element.

\begin{theorem}[Deterministic Farey-boundary reduction and complexity]
	\label{thm:smashed-boundary-algorithm}
	Given a cyclic prefix table of size $s$ for $t\in\mathsf T$, the
	deterministic marked fan reduction has the following properties.
	\begin{enumerate}
		\item It computes flips $\psi_1,\ldots,\psi_{N(t)}$, each of the form
		$\phi_I$, $\phi_{UL}$ or $\phi_{UL^2}$ with $U$ admissible, and the reduced
		word in $S,L,L^2$ of an element $m_t\in\psl$ such that
		\begin{equation}\label{eq:smashed-boundary-algorithm-output}
			t=m_t\,\psi_{N(t)}^{-1}\cdots\psi_1^{-1}.
		\end{equation}
		After substituting the fixed flip words for $S,L,L^2$, the right-hand
		side is a word entirely in left and right boundary flips.
		\item The unmarked fan stage is within a factor two of an optimal
		elementary boundary reduction:
		\begin{equation}\label{eq:main-factor-two-bound}
			k_-(t)=k_+(t),\qquad
			k_-(t)\leq N_{\min}(t)\leq N(t)
			\leq k_-(t)+k_+(t)=2k_-(t).
		\end{equation}
		\item The complete output contains $O(s)$ flip factors, has explicitly
		encoded length $O(s^2)$, and is computed in $O(s^3)$ time under the
		computational convention above.
	\end{enumerate}
	The output word and the residual element depend only on $t$ and the fixed
	ordering conventions, not on the prefix table chosen to compute them.
\end{theorem}

The factor bound counts generators in the address-indexed
flip family.  It is distinct from a word-length bound in a
fixed finite generating set; the encoded-length bound also
charges for writing each supporting address.

\begin{proof}
	\emph{Disagreement blocks.}  By \cref{lem:effective-finite-support}, compute
	$\mathscr E_t\setminus\mathscr E_0$ and $\mathscr E_0\setminus\mathscr E_t$.
	For $e\in\mathscr E_t\setminus\mathscr E_0$, its two common neighbours are
	obtained by applying $t$ to the two standard common neighbours of
	$t^{-1}(e)\in\mathscr E_0$.  Hence \cref{lem:finite-block-decomposition}
	computes the disagreement blocks.  Order the blocks by the $\prec$-least member of
	$\mathscr E_0[C]$. These keys are distinct, since the nonempty
	sets $\mathscr E_0[C]$ partition
	$\mathscr E_0\setminus\mathscr E_t$.
	For each block, give its $\prec$-least boundary pair its
	preferred orientation and take its initial endpoint as the fan base. In the
	block fan reduction, whenever there is more than one possible gap between
	consecutive neighbours of the base point, choose the first gap in the cyclic
	order beginning with the preferred boundary pair.  These conventions make every
	choice depend only on $\mathscr E_t$ and the fixed order
	\eqref{eq:farey-edge-order}.
	
	\emph{Fan stage.}  In each block $C$, first use the block fan reduction to
	carry the restriction of $\mathscr E_t$ to the chosen fan.  Apply the same
	construction to the restriction of $\mathscr E_0$, and reverse this second
	sequence.  Every exchange fixes $\partial C$, so it changes no other block and
	no pair outside $C$.  Processing all blocks gives pairs $e_1,\ldots,e_N$, where
	$N=N(t)$, and Farey-type boundary patterns
	\[
	\mathscr E^{(0)}=\mathscr E_t,\qquad
	\mathscr E^{(j)}=D_{e_j}\bigl(\mathscr E^{(j-1)}\bigr)
	\quad(e_j\in\mathscr E^{(j-1)}),\qquad
	\mathscr E^{(N)}=\mathscr E_0.
	\]
	
	\emph{Flips.}  Put $\theta_0:=t$, so that $\theta_0(\mathscr E_0)=\mathscr E^{(0)}$.
	If $\theta_{j-1}\in\mathsf T$ satisfies
	$\theta_{j-1}(\mathscr E_0)=\mathscr E^{(j-1)}$, put
	\[
	e_j^\circ:=\theta_{j-1}^{-1}(e_j)\in\mathscr E_0,\qquad
	\psi_j:=\Phi(e_j^\circ),\qquad
	\theta_j:=\theta_{j-1}\psi_j.
	\]
	By \cref{lem:exchange-flip},
	\[
	\theta_j(\mathscr E_0)
	=\theta_{j-1}\bigl(D_{e_j^\circ}(\mathscr E_0)\bigr)
	=D_{e_j}\bigl(\theta_{j-1}(\mathscr E_0)\bigr)
	=\mathscr E^{(j)}.
	\]
	The endpoints of $e_j^\circ$ are obtained by applying the table of
	$t^{-1}$ and then $\psi_1^{-1},\ldots,\psi_{j-1}^{-1}$ to the
	endpoints of $e_j$.  If $e_j^\circ=\alpha_I$, then $\psi_j=\phi_I$.
	
	Otherwise, write $e_j^\circ=\langle x,y\rangle$.  Each of
	$\bar\kappa(x)$ and $\bar\kappa(y)$ has two eventually constant
	binary representatives.  Among the four choices of one
	representative for each endpoint, choose the pair whose longest
	common prefix has greatest length, and denote that prefix by $w$.
	This maximizing pair is unique.  Indeed, if $W$ is the unique
	admissible prefix with $\alpha_W=e_j^\circ$, and $w_W$ is its
	binary code, the two representatives lying in its cylinder are
	$w_W0^\infty$ and $w_W1^\infty$.  Their longest common prefix is
	$w_W$, whereas every other choice gives a shorter common prefix.
	Thus $w=w_W$.
	
	Translate $w$ into the admissible Farey prefix $W$.  Since
	$e_j^\circ\neq\alpha_I$, uniquely write $W=UB$ or $W=UA$.
	Then \cref{lem:exchange-flip} gives
	\[
	\psi_j=
	\begin{cases}
		\phi_{UL},& W=UB,\\
		\phi_{UL^2},& W=UA.
	\end{cases}
	\]
	This explicitly computes the previously defined
	$\psi_j=\Phi(e_j^\circ)$.
		
	\emph{Residual element.}  The element
	$m_t:=\theta_N=t\psi_1\cdots\psi_N$ satisfies $m_t(\mathscr E_0)=\mathscr E_0$.
	Since the stabilizer of $\mathscr E_0$ in $\mathsf T$ is $\psl$, we have
	$m_t\in\psl$.  Compute $m_t(\vec e_0)=(m_t(\infty),m_t(0))$ by applying
	$\psi_N,\ldots,\psi_1$ and then the table of $t$ to $\infty$ and $0$.  Write
	$m_t(\vec e_0)=\left(\frac ac,\frac bd\right)$, choosing primitive integer
	representatives with $ad-bc=1$.  The matrix
	\[
	M=\begin{pmatrix}a&b\\ c&d\end{pmatrix}\in\sll
	\]
	sends $\vec e_0=(\infty,0)$ to $m_t(\vec e_0)$, so, by simple transitivity of
	the $\psl$-action on oriented Farey edges, its class in $\psl$ is $m_t$.
	Applying the Euclidean algorithm to $M$, followed by reduction using
	$\psl=\langle S,L\mid S^2=L^3=1\rangle$, produces the reduced free-product
	normal form of $m_t$.  Finally, $t\psi_1\cdots\psi_N=m_t$ gives
	\eqref{eq:smashed-boundary-algorithm-output}, and
	\eqref{eq:modular-generators-as-flips} expresses it entirely in boundary flips.
	
	For a disagreement block $C$ with $r$ points, the two
	fan reductions together use at most $2(r-3)$ exchanges.
	By \eqref{eq:defect-cardinalities-equal},
	\[
	N(t)\leq2\sum_C(|C|-3)=2k_-(t),
	\qquad k_-(t)=k_+(t).
	\]
	Every missing standard pair must be inserted, and each
	exchange inserts only one pair. Hence
	$k_-(t)\leq N_{\min}(t)\leq N(t)$, proving
	\eqref{eq:main-factor-two-bound}.
	
	The fixed ordering conventions determine the exchanges
	from $\mathscr E_t$. Starting with $\theta_0=t$, the
	recursion then determines every $\psi_j$ and $m_t$.
	Thus the output is independent of the input table.
	
	The complete output and running-time bounds follow from
	\eqref{eq:complete-bound-lambda},
	\cref{lem:residual-modular-length,cor:T-input-complexity}
	below.
\end{proof}

\begin{remark}[Alternative linked-pair insertion]
	One may instead insert missing standard pairs successively.
	Give a missing pair $\alpha$ its preferred orientation
	$(x,y)$. Among the pairs $\langle u,v\rangle\in\mathscr E$
	linked with $\alpha$, with $u\in(x,y)$ and $v\in(y,x)$,
	choose the least $u$ and, among ties, the greatest $v$.
	Any member of $\mathscr E$ linked with $\langle x,u\rangle$
	or $\langle x,v\rangle$ would have its first endpoint
	before $u$, or share $u$ and have its second endpoint
	after $v$, contradicting this choice.
	Maximality therefore makes $x$ a common neighbour of
	$u,v$. If $z$ is their other common neighbour, exchange
	$\langle u,v\rangle$ for $\langle x,z\rangle$.
	This decreases the finite number of pairs linked with
	$\alpha$ by one; repetition inserts $\alpha$ by maximality.
	No standard pair is removed, and the number of nonstandard
	pairs never increases. Starting from $\mathscr E_t$, the
	procedure therefore uses at most $k_-(t)k_+(t)$ exchanges.
\end{remark}

We next include the marking in the estimate.  Let
\[
m_t=t\psi_1\cdots\psi_{N(t)}\in\psl
\]
be the residual modular element of \cref{thm:smashed-boundary-algorithm}, and
write its reduced word in $S$, $L$ and $L^{2}$ as
\begin{equation}\label{eq:residual-modular-normal-form}
	m_t=u_1\cdots u_q,
	\qquad u_j\in\{S,L,L^2\}.
\end{equation}

Let $q_S(t)$ be the number of letters $S$ in this word and let $q_3(t)$ be
the number of letters $L$ or $L^2$; thus
$q=q_S(t)+q_3(t)$.  Here and below, the length of an output word counts each flip $\phi_I^{\pm1}$, $\phi_{UL}$ or $\phi_{UL^2}$ as one factor; in particular, $\phi_I^{-1}=\phi_I^3$ counts as one factor.

Let $\Lambda_{\mathsf T}(t)$ count the flip factors after
substituting the modular-generator words.
Each inverse fan-stage flip contributes one factor,
while \eqref{eq:modular-generators-as-flips} contributes
two for $S$ and five for $L$ or $L^2$. Hence
\begin{equation}\label{eq:complete-bound-lambda}
	\Lambda_{\mathsf T}(t)
	\leq N(t)+2q_S(t)+5q_3(t)
	\leq k_-(t)+k_+(t)+5q,
\end{equation}
where the second inequality uses
\eqref{eq:main-factor-two-bound}.

%
%

For $x\in\mathscr R$ let $\ell(x)$ be the least length of a binary word $u$ such
that $u0^\infty$ or $u1^\infty$ represents $\bar\kappa(x)$.

\begin{lemma}[Address growth]
	\label{lem:address-growth}
	Let $x\in\mathscr R$.
	\begin{enumerate}
		\item If $\phi$ is a flip, then $\ell(\phi(x))\leq\ell(x)+2$.
		\item If $h=h_{P,Q,\tau} \in \T$, then $\ell(h(x))\leq\ell(x)+\max_i|q_i|$.
		\item If $W$ is an admissible prefix with binary code $w$, then the
		larger of $\ell(W(0))$ and $\ell(W(\infty))$ equals $|w|$, and the
		reduced word of $W$ in $S,L,L^2$ has at most $2|w|-1$ letters.
	\end{enumerate}
\end{lemma}

\begin{proof}
	Let $ud^\infty$, with $d\in\{0,1\}$ and $|u|=\ell(x)$, represent $x$.
	
	(1) In binary coordinates, $\phi_{WL}$ acts on $w\mathcal B$ by
	$w0\xi\mapsto w00\xi$, $w10\xi\mapsto w01\xi$, $w11\xi\mapsto w1\xi$, and
	$\phi_{WL^2}$ is its inverse; both fix the complement of $w\mathcal B$.
	Thus a prefix $p$ is replaced by a prefix $q$ with $|q|\leq|p|+1$.  If
	$|u|\geq|p|$, the image of $ud^\infty$ has a representative $vd^\infty$ with
	$|v|\leq|u|+1$.  If $|u|<|p|$, then $p=ud^{\,|p|-|u|}$, and $ud^\infty$ is
	fixed except for $w10\,0^\infty\mapsto w01\,0^\infty$ and
	$w01\,1^\infty\mapsto w10\,1^\infty$, where $\ell$ grows by one.  The base
	flip $\phi_I^{\pm1}$ replaces prefixes of length $2$ by prefixes of length
	$2$, so $\ell$ grows by at most two.
	
	(2) If $ud^\infty\in p_i\mathcal B$, then $h(ud^\infty)=q_{\tau(i)}vd^\infty$
	with $|v|=\max\{|u|-|p_i|,0\}$.
	
	(3) Write $W=S^\delta C_1\cdots C_k$ with $C_i\in\{A,B\}$, so that
	$|w|=k+1$.  The points $W(0)$ and $W(\infty)$ are represented by $w0^\infty$
	and $w1^\infty$.  Hence $\ell(W(0))$ and $\ell(W(\infty))$ are the lengths of
	$w$ with its terminal $0$'s, respectively $1$'s, removed, and one of them is
	$|w|$.  The reduced word $S^\delta L^{i_1}S\cdots L^{i_k}S$ has $\delta+2k\leq2|w|-1$ letters.
\end{proof}

\begin{lemma}[Residual modular length]
	\label{lem:residual-modular-length}
	If the input prefix table has size $s=s(P,Q)$, then
	\[
	q\leq 2s+4N(t)+1.
	\]
\end{lemma}

\begin{proof}
	We have $\ell(\infty)=1$ and $\ell(0)=0$.  By
	\cref{lem:address-growth}(1),(2), the endpoints of
	$m_t(\vec e_0)=t\psi_1\cdots\psi_{N(t)}(\vec e_0)$ satisfy
	$\ell\leq s+2N(t)+1$.  Write $m_t\in\{W,WS\}$ with $W$ admissible, as in the
	proof of \cref{lem:exchange-flip}; then
	$\alpha_W=\langle m_t(0),m_t(\infty)\rangle$.  By
	\cref{lem:address-growth}(3), the binary code $w$ of $W$ satisfies
	$|w|\leq s+2N(t)+1$, and $W$ has at most $2|w|-1$ letters.  The same bound
	holds for $WS$, which is obtained by deleting the final $S$ of $W$, or equals
	$S$ if $W=I$.  Hence $q\leq2|w|-1\leq2s+4N(t)+1$.
\end{proof}

\begin{corollary}[Input-size complexity]
	\label{cor:T-input-complexity}
	For a cyclic prefix table of size $s$, the deterministic marked reduction
	produces $O(s)$ flip factors.  Its explicitly encoded output has length
	$O(s^2)$, and a direct implementation on finite $A$--$B$ prefix words has
	running time $O(s^3)$.
\end{corollary}

\begin{proof}
	By \eqref{eq:defect-table-bound} and
	\eqref{eq:main-factor-two-bound},
	$N(t)\leq k_-(t)+k_+(t)\leq2s$.
	By \cref{lem:residual-modular-length}, $q\leq10s+1$.
	Hence \eqref{eq:complete-bound-lambda} yields
	\begin{equation}\label{eq:complete-linear-table-bound}
		\Lambda_{\mathsf T}(t)\leq N(t)+5q=O(s).
	\end{equation}
	
	Every point handled during the computation has $\ell=O(s)$.  Indeed, the
	endpoints of the pairs in $\mathscr A(P)\cup\mathscr A(Q)$ from the proof of \cref{lem:effective-finite-support} and their standard common neighbours have $\ell\leq s+1$; by \cref{lem:address-growth}, applying
	$t^{\pm1}$ adds at most $s$, each flip adds at most $2$, and at most
	$N(t)\leq2s$ flips are applied to any point.  By
	\cref{lem:address-growth}(3), every $\psi_j$ has a supporting address of
	length $O(s)$, while the modular part uses only $\phi_I$, $\phi_{SL}$ and
	$\phi_{SL^2}$.  Hence the encoded output length is
	\begin{equation}\label{eq:T-encoded-output-bound}
		O\bigl(N(t)\,s+q\bigr)=O(s^2).
	\end{equation}
	
	The initial finite-defect computation also takes $O(s^3)$
	time. There are $O(s)$ candidate pairs in $\mathscr A(P)\cup\mathscr A(Q)$, and their endpoints
	and required standard common neighbours have addresses
	of length $O(s)$ and integer coordinates of $O(s)$ bits.
	Each required conversion, Farey-condition test, and
	standard common-neighbour computation takes $O(s^2)$
	time with schoolbook arithmetic.
	
	For the running time, there are $O(s)$ such pairs, defect pairs and
	block incidences, each described by words of length $O(s)$.  Constructing
	and ordering the incidence forest uses $O(s^2)$ comparisons of such words,
	hence $O(s^3)$ time.  The fan stage performs $O(s)$ exchanges, each scanning
	$O(s)$ pairs of its block, which costs $O(s^3)$.  Computing $e_j^\circ$
	applies $t^{-1}$ and at most $j-1$ flips to two points with $\ell=O(s)$,
	each application costing $O(s)$; all pull-backs cost
	$O(N(t)^2s)=O(s^3)$.  Computing $m_t(\vec e_0)$ costs $O(N(t)s)$.  Each turn
	at most doubles the entries of the corresponding matrix, so the Euclidean
	reduction involves integers of $O(s)$ bits and costs $O(s^2)$ with
	schoolbook arithmetic; writing the output costs $O(s^2)$.  Thus the total
	running time is $O(s^3)$, without a unit-cost assumption on long addresses
	or rational boundary points.
\end{proof}

The term involving $q$ in \eqref{eq:complete-bound-lambda} is necessary:
for $t\in\psl$ the boundary defect vanishes, whereas the reduced modular
length of $t$ is unbounded.

\begin{example}\label{ex:marked-fan-reduction}
	Consider the cyclic prefix table
	\[
	\begin{array}{c|ccccc}
		\text{source}&000&0010&0011&01&1\\ \hline
		\text{target}&100&101&11&00&01
	\end{array}
	\]
	for $t=\phi_{SL}L$.  Its boundary defect is
	\[
	\mathscr E_t\setminus\mathscr E_0
	=\{\langle-2,0\rangle\},
	\qquad
	\mathscr E_0\setminus\mathscr E_t
	=\{\langle-1,\infty\rangle\}.
	\]
	There is one disagreement block, with cyclic list
	$(\infty,-2,-1,0)$. 
	Its $\prec$-least boundary pair is $\langle\infty,0\rangle$, with preferred orientation $(\infty,0)$, so the fan base is $\infty$.
	The fan stage consists of the single exchange
	\[
	e_1=\langle-2,0\rangle
	\longmapsto\langle-1,\infty\rangle.
	\]
	
	To extract its flip, pull back the exchanged pair:
	\[
	e_1^\circ=t^{-1}(e_1)
	=\left\langle\frac12,1\right\rangle
	=\alpha_{BA}.
	\]
	Thus $\psi_1=\phi_{BL^2}$ and
	\[
	m_t=t\psi_1=L.
	\]
	The algorithm therefore returns
	\[
	t=L\phi_{BL}.
	\]
	Substituting \eqref{eq:modular-generators-as-flips} gives a
	word entirely in boundary flips.  In particular, the flip
	recorded by the algorithm is determined by the pulled-back
	pair $e_1^\circ$, not by interpreting the diagonal exchange
	itself as a fixed boundary homeomorphism.
\end{example}

\subsection{Elements of \texorpdfstring{$\mathsf V$}{V} as sequences of
flips and shuffles}
For a cyclic prefix-replacement map $g\in\mathsf T$, let $\mathcal A_{\mathsf T}(g)$ denote the flip word returned by
\cref{thm:smashed-boundary-algorithm}.  The construction for $\mathsf V$
below is an extension of that boundary algorithm rather than a second
boundary-reduction procedure: the cyclic factors are handled by
$\mathcal A_{\mathsf T}$, and only the remaining finite cylinder permutation
is treated by standard adjacent-transposition sorting.

\begin{theorem}[Flip--shuffle reduction]
\label{prop:V-effective-algorithm}
Given a finite prefix table for $v\in\mathsf V$, there is a finite algorithm which returns
a finite word in flips and the basic shuffle $\sigma_S$ representing $v$.
Apart from finite permutation sorting and word operations, it invokes
$\mathcal A_{\mathsf T}$ three times.
\end{theorem}

\begin{proof}
Write the input table as
\[
 v=h_{P,Q,\tau},\qquad
 P=\{p_1<\cdots<p_n\},\qquad
 Q=\{q_1<\cdots<q_n\},
\]
so that
\[
 v(p_i\xi)=q_{\tau(i)}\xi.
\]
If $n<3$, repeatedly expand the row with lexicographically
least source prefix until the table has three rows.
Henceforth $P,Q,\tau$ and $n$ refer to this expanded table.
This preliminary step increases the input size by at most
an absolute constant.

Use the canonical maximal prefix code
\[
 C_n=\{c_1<\cdots<c_n\},\qquad
 c_j=1^{j-1}0\quad(1\leq j<n),\qquad
 c_n=1^{n-1}.
\]
For example, $C_3=\{0,10,11\}$ and
$C_4=\{0,10,110,111\}$.  Define the cyclic prefix replacements
\[
 a_P(c_i\xi)=p_i\xi, \qquad
 a_Q(c_i\xi)=q_i\xi.
\]
They belong to $\mathsf T$.  Define also the cylinder permutation
\[
 \pi_\tau(c_i\xi)=c_{\tau(i)}\xi.
\]
Evaluation on each $p_i\mathcal B$ gives the factorization
\begin{equation}\label{eq:V-canonical-factorization}
 v=a_Q\pi_\tau a_P^{-1}.
\end{equation}
Compute
\[
 A_P:=\mathcal A_{\mathsf T}(a_P),\qquad
 A_Q:=\mathcal A_{\mathsf T}(a_Q).
\]

Next let $\rho_n\in\mathsf T$ be the cyclic shift of $C_n$,
\[
 \rho_n(c_i\xi)=c_{i+1}\xi\quad(1\leq i<n),
 \qquad
 \rho_n(c_n\xi)=c_1\xi,
\]
and make the third and final call
\[
 R_n:=\mathcal A_{\mathsf T}(\rho_n).
\]
Let $W_n$ be the admissible Farey prefix whose binary code is
$1^{n-2}$; in particular, $W_3=S$.  The shuffle
\[
 z_n:=\sigma_{W_n}
\]
exchanges the last two cylinders
$c_{n-1}\mathcal B$ and $c_n\mathcal B$ and fixes all other cylinders of
$C_n$.  It can be written using only the basic shuffle $\sigma_S$.  Indeed,
$X_0^{-(n-3)}$ sends $10\mathcal B$ and $11\mathcal B$ to those last two
cylinders, and therefore
\begin{equation}\label{eq:canonical-last-shuffle}
 z_n=X_0^{-(n-3)}\sigma_SX_0^{n-3}.
\end{equation}
Here $X_0$ is already the explicit flip word in
\eqref{eq:flip-generators}.

Let $s_i=(i\ i+1)$ for $1\leq i<n$, acting on the cylinders of $C_n$, and
choose the representative $k_i\in\Z$ of least absolute value satisfying
\[
 k_i\equiv i-(n-1)\pmod n.
\]
If the two representatives have the same absolute value, choose the
nonnegative one.  Thus $2|k_i|\leq n$.
Conjugating \eqref{eq:canonical-last-shuffle} by the appropriate cyclic
shift gives
\begin{equation}\label{eq:canonical-adjacent-transposition}
 s_i=\rho_n^{k_i}z_n\rho_n^{-k_i}.
\end{equation}

Finally, use bubble sort to write
\[
 \tau=s_{i_r}\cdots s_{i_1}.
\]
Here $r=\operatorname{inv}(\tau)\leq n(n-1)/2$.  Equations
\eqref{eq:canonical-last-shuffle} and
\eqref{eq:canonical-adjacent-transposition} give the explicit word
\begin{equation}\label{eq:V-algorithm-output}
 \boxed{
 \begin{aligned}
 v={}&A_Q
 \bigl(R_n^{k_{i_r}}z_nR_n^{-k_{i_r}}\bigr)
 \cdots
 \bigl(R_n^{k_{i_1}}z_nR_n^{-k_{i_1}}\bigr)
 A_P^{-1},\\
 z_n={}&X_0^{-(n-3)}\sigma_SX_0^{n-3}.
 \end{aligned}}
\end{equation}
Every symbol on the right side is now a flip or the shuffle $\sigma_S$.
The factorization \eqref{eq:V-canonical-factorization} proves correctness,
and termination follows from the termination of
$\mathcal A_{\mathsf T}$ and the finiteness of the sorting procedure.
The use of finite cylinder permutations is standard in the prefix-code
description of $\mathsf V$; see
\cite[\S6]{CannonFloydParry1996} and
\cite[\S2.1]{BleakQuick2017}.
\end{proof}

For an input table $v=h_{P,Q,\tau}$, its \emph{$\mathsf V$-input size} is $s_{\mathsf V}:=s(P,Q)$, as in \eqref{eq:prefix-table-size}.
Count every flip and every occurrence of the basic shuffle $\sigma_S$ as
one factor, and let $\Lambda_{\mathsf V}(v)$ denote the number of factors in
the word \eqref{eq:V-algorithm-output}.

In the following estimates, $n$ and $\tau$ are those used
after the preliminary expansion, while $s_{\mathsf V}$
denotes the size of the original input table.

\begin{proposition}[Complete output and running time for $\mathsf V$]
	\label{prop:V-complete-complexity}
	Let $r=\operatorname{inv}(\tau)$.
	
	\begin{enumerate}
		\item[\textup{(i)}] \emph{Output bound.} The word produced by the
		flip--shuffle reduction satisfies
		\begin{equation}\label{eq:V-complete-length-bound}
			\Lambda_{\mathsf V}(v)
			\leq
			\Lambda_{\mathsf T}(a_P)+\Lambda_{\mathsf T}(a_Q)
			+r\bigl(
			n\Lambda_{\mathsf T}(\rho_n)+10(n-3)+1
			\bigr).
		\end{equation}
		Consequently,
		\[
		\Lambda_{\mathsf V}(v)=O(s_{\mathsf V}^{5}).
		\]
		If every supporting address is written explicitly, the encoded output
		length is $O(s_{\mathsf V}^{7})$.
		
		\item[\textup{(ii)}] \emph{Running time.} A direct implementation which
		writes the fully expanded output runs in
		\[
		O(s_{\mathsf V}^{7})
		\]
		time.  In particular, the flip--shuffle reduction is polynomial-time.
	\end{enumerate}
\end{proposition}

\begin{proof}
	\emph{(i)}
	The word for $X_0$ has five flip factors, so
	$z_n=X_0^{-(n-3)}\sigma_SX_0^{n-3}$ has
	$10(n-3)+1$ factors. Since $2|k_i|\leq n$, each
	adjacent transposition contributes at most
	\[
	n\Lambda_{\mathsf T}(\rho_n)+10(n-3)+1
	\]
	factors. Summing over the $r$ transpositions and adding
	the two outside words proves
	\eqref{eq:V-complete-length-bound}.
	
	The total prefix length of $C_n$ is $O(n^2)$.
	Thus the three calls to $\mathcal A_{\mathsf T}$ have
	input size $O(s_{\mathsf V}+n^2)=O(s_{\mathsf V}^2)$.
	By \cref{cor:T-input-complexity}, each returns
	$O(s_{\mathsf V}^2)$ factors and
	$O(s_{\mathsf V}^4)$ encoded symbols.
	Since $r=O(s_{\mathsf V}^2)$ and
	$n=O(s_{\mathsf V})$, the factor bound gives
	$\Lambda_{\mathsf V}(v)=O(s_{\mathsf V}^5)$.
	There are at most $rn=O(s_{\mathsf V}^3)$ copies
	of $R_n$ or its inverse in the expanded word, giving
	encoded length $O(s_{\mathsf V}^7)$; the outside
	words and copies of $z_n$ contribute lower-order terms.
	
	\smallskip
	\emph{(ii)}
	Table construction and permutation sorting take at most
	$O(s_{\mathsf V}^3)$ bit operations. The three calls
	to $\mathcal A_{\mathsf T}$ take $O(s_{\mathsf V}^6)$
	time altogether. Expanding and writing the resulting
	word takes $O(s_{\mathsf V}^7)$ time, which dominates
	the preceding computations.
\end{proof}


\section{Boundary presentations as structural applications}
\label{sec:boundary-presentations}

The algorithms above use the boundary realizations of the generators, but not
a presentation of their relations.  We now record the presentation results as
structural applications of the same boundary dictionary.  They are explicit
Tietze forms of known presentations rather than new abstract presentations.

\begin{remark}\label{rem:dehornoy-geometric-presentation}
	Under the binary coding $B\leftrightarrow0$ and $A\leftrightarrow1$,
	the identities in \cref{prop:flip-shuffle-relations} correspond to
	Dehornoy's associativity and commutativity relations at nonempty
	addresses, with product order reversed to match our composition
	convention. Including the empty-address instances gives his
	presentation of $\mathsf V$
	\cite[Proposition~3.18]{Dehornoy2005}.
	The finite presentation in \cref{thm:V-flip-shuffle-presentation}
	is obtained by an explicit change of generators from the
	Bleak--Quick presentation
	\cite[Theorem~2.1]{BleakQuick2017}.
\end{remark}

\begin{theorem}[A presentation of $\mathsf T$ by boundary flips]
	\label{thm:T-flip-presentation}
	Put
	\[
	 \rho:=\phi_I\phi_{SL^2}\phi_I^3\phi_{SL^2}\phi_I,
	 \qquad
	 \omega:=\rho\phi_I.
	\]
	In the free group on $\phi_I$ and $\phi_{SL^2}$, use the abbreviation
	$\phi_{SL}:=(\phi_{SL^2})^{-1}$.
	Then Thompson's group $\mathsf T$ has the finite presentation
	\begin{equation}\label{eq:T-direct-flip-presentation}
		\mathsf T\cong
		\left\langle
		\phi_I,\phi_{SL^{2}}
		\ \middle|\
		\begin{aligned}
			&\phi_I^4=1,\qquad
			\rho^3=1,\qquad
			\omega^5=1,\\
			&\phi_{SL}
			=\phi_I^2\rho\phi_I\rho\phi_I^2,\\
			&[\rho\phi_I\rho,\phi_{SL}]=1,\\
			&[\rho\phi_I\rho,
			\phi_I^2\rho\phi_{SL}\rho^2\phi_I^2]=1
		\end{aligned}
		\right\rangle .
	\end{equation}
\end{theorem}

\begin{proof}
	A direct calculation from the prefix tables gives
	\[
	\rho
	=\phi_I\phi_{SL^2}\phi_I^3\phi_{SL^2}\phi_I
	=L^2.
	\]
	Start with the Lochak--Schneps presentation with generators
	$\alpha$ and $\beta$, using the conjugate relator
	$(\beta\alpha)^5=1$; see
	\cite[Theorem~1]{LochakSchneps1997}.  Adjoin a generator
	$\phi_{SL^2}$ subject to
	\[
	(\phi_{SL^2})^{-1}
	=\alpha^2\beta\alpha\beta\alpha^2.
	\]
	The prefix-table identity
	\[
	\beta
	=\alpha\phi_{SL^2}\alpha^{-1}\phi_{SL^2}\alpha
	\]
	then permits the elimination of $\beta$.  Using $\alpha^4=1$, the
	eliminated word is precisely $\rho$, while
	$\beta\alpha$ becomes $\omega$.  The two commutator relators become
	\[
	[\rho\phi_I\rho,\phi_{SL}]=1
	\]
	and
	\[
	[\rho\phi_I\rho,
	\phi_I^2\rho\phi_{SL}\rho^2\phi_I^2]=1.
	\]
	Thus \eqref{eq:T-direct-flip-presentation} is obtained from the
	Lochak--Schneps presentation by Tietze transformations and is therefore
	complete.
\end{proof}

Define the following abbreviations:
\begin{align*}
	\chi&=(\omega\phi_I)\phi_I^2(\omega\phi_I),\\
	\alpha&=\phi_I^{-2}\sigma_S\phi_I^2,
	&\beta&=\phi_I^{-1}\alpha,\\
	\gamma&=\chi\sigma_S\chi^{-1},
	&D&=\alpha\gamma\alpha\gamma\phi_I^2.
\end{align*}
In the relations below, $x^y$ means $y^{-1}xy$.

\begin{theorem}[A finite presentation of $\mathsf V$
	by boundary flips and shuffles]
	\label{thm:V-flip-shuffle-presentation}
	Thompson's group $\mathsf V$ has a presentation with generators
	\[
	\phi_I,\qquad \phi_{SL^2},\qquad \sigma_S
	\]
	and the following nine relations:
	\begin{align}
		&\phi_I^4=1,\qquad
		\sigma_S^2=1,\qquad
		\beta^3=1,
		\label{eq:V-finite-torsion}\\
		&\gamma^{\alpha\gamma}=\alpha,
		\label{eq:V-finite-conjugacy}\\
		&\gamma
		=\alpha^{\beta\gamma\alpha\gamma\alpha\alpha^{\beta\alpha}}
		\alpha^{\beta^{-1}\gamma\alpha\gamma\alpha
			\alpha^{\beta^{-1}\alpha}},
		\label{eq:V-finite-splitting}\\
		&\left[
		\alpha^{\beta\gamma\alpha\gamma},
		\alpha^{\beta^{-1}\gamma\alpha\gamma
			\alpha^\beta\alpha^{\beta^{-1}\alpha}}
		\right]=1,
		\label{eq:V-finite-commutation-one}\\
		&\left[
		\alpha^{\beta^{-1}\gamma\alpha\gamma},
		\alpha^{\beta^{-1}\gamma\alpha\gamma
			\alpha^\beta\alpha^{\beta^{-1}\alpha}}
		\right]=1,
		\label{eq:V-finite-commutation-two}\\
		&\left[
		\alpha^{\beta\gamma\alpha^{\beta\gamma\alpha}},
		\alpha^{\beta\gamma\alpha\gamma
			\alpha^\beta\alpha^{\beta^{-1}\alpha}}
		\right]=1,
		\label{eq:V-finite-commutation-three}\\
		&\phi_{SL^2}
		=D^{-1}\phi_I D^{-1}\phi_I^{-2}.
		\label{eq:V-finite-flip-recovery}
	\end{align}
\end{theorem}

\begin{proof}
	Start with the three-generator, eight-relation presentation
	$\langle a,b,c\mid\mathcal R_{\mathrm{BQ}}\rangle$
	of \cite[Theorem~2.1, equation~(2.3)]{BleakQuick2017}.
	Passing from their right actions to our left-composition convention
	by taking inverse maps, the generators have the following actions:
	$a$ exchanges $00\xi$ and $01\xi$;
	$b$ sends
	\[
	01\xi\longmapsto11\xi\longmapsto10\xi\longmapsto01\xi;
	\]
	and $c$ exchanges $1\xi$ and $00\xi$.
	Each fixes the complement of the indicated cylinders.
	
	Put
	\[
	p=ab^{-1},\qquad
	z=p^2ap^{-2},\qquad
	d=acacp^2,\qquad
	q=d^{-1}pd^{-1}p^{-2}.
	\]
	Direct composition of these prefix replacements gives
	\[
	p=\phi_I,\qquad z=\sigma_S,\qquad
	d=X_0,\qquad q=\phi_{SL^2}.
	\]
	In particular, \eqref{eq:flip-generators} gives $X_0=\rho\phi_I^2=\omega\phi_I$, so $d=\omega p$, with $\omega$ evaluated in $p,q$.
	
	Conversely,
	\[
	a=p^{-2}zp^2,\qquad b=p^{-1}a.
	\]
	The prefix table
	\[
	dp^2d:\quad
	0\xi\mapsto01\xi,\qquad
	10\xi\mapsto1\xi,\qquad
	11\xi\mapsto00\xi
	\]
	also gives
	\[
	c=(dp^2d)\,z\,(dp^2d)^{-1}.
	\]
	Thus $a,b,c$ are precisely the words
	$\alpha,\beta,\gamma$ displayed above.
	
	Adjoin $p,z,q$ to the Bleak--Quick presentation by their defining
	relations. The recovery identities permit the elimination of
	$a,b,c$. Their eight relators become the relations
	\eqref{eq:V-finite-torsion}--%
	\eqref{eq:V-finite-commutation-three}:
	indeed, $a^2=1$ becomes $z^2=1$, and, using this relation,
	$(ab)^4=1$ becomes $p^4=1$.
	The defining relation for $q$ becomes
	\eqref{eq:V-finite-flip-recovery}, since $d$ becomes $D$.
	These generator additions and eliminations preserve the presented
	group. Hence the nine displayed relations present $\mathsf V$.
\end{proof}



\section*{Acknowledgements}
The author thanks IHES for their hospitality where most of this work was written. This work was supported by Galatasaray Research Grant FBA-2024-1236.

\end{document}